\documentclass[letterpaper]{article}
\usepackage[margin=1in]{geometry}
\usepackage[utf8]{inputenc} 
\usepackage[T1]{fontenc}    

\usepackage{palatino}
\usepackage{natbib}
\usepackage{url}            
\usepackage{booktabs}       
\usepackage{amsmath,amsthm,amsfonts}       
\usepackage{nicefrac}       
\usepackage{microtype}      

\usepackage[usenames,dvipsnames]{xcolor}   
\usepackage{hyperref}
\definecolor{darkblue}{rgb}{0.0,0.0,0.65}
\definecolor{darkred}{rgb}{0.68,0.05,0.0}
\definecolor{darkgreen}{rgb}{0.0,0.29,0.29}
\definecolor{darkpurple}{rgb}{0.47,0.09,0.29}
\hypersetup{
   colorlinks = true,
   citecolor  = darkblue,
   linkcolor  = darkred,
   filecolor  = darkblue,
   urlcolor   = darkblue,
 }

\title{\bfseries Invex Programs: First Order Algorithms and Their Convergence}
\date{}

\author{%
  Adarsh Barik\\
  Department of Computer Science\\
  Purdue University\\
  West Lafayette, IN 47906\\
  \texttt{abarik@purdue.edu} \\
  \and
  Suvrit Sra\\
  Electrical Engineering \& Computer Science Department\\
  Massachusetts Institute of Technology (MIT)\\
  Cambridge, MA 02139\\
  \texttt{suvrit@mit.edu} \\
  \and
  Jean Honorio\\
  Department of Computer Science\\
  Purdue University\\
  West Lafayette, IN 47906\\
  \texttt{jhonorio@purdue.edu} \\
}

\usepackage{subcaption}

\usepackage{algorithm}
\usepackage{algorithmic}

\usepackage{amsmath}
\usepackage{amssymb}
\usepackage{mathtools}
\usepackage{amsthm}

\usepackage{tikz}

\usepackage[capitalize,noabbrev]{cleveref}

\newtheorem{theorem}{Theorem}[section]
\newtheorem{proposition}[theorem]{Proposition}
\newtheorem{lemma}[theorem]{Lemma}

\theoremstyle{definition}
\newtheorem{definition}[theorem]{Definition}
\newtheorem{assumption}[theorem]{Assumption}

\newtheorem{example}{Example}

\usepackage[textsize=tiny]{todonotes}

\newcommand{\calC}{\mathcal{C}}
\newcommand{\calA}{\mathcal{A}}
\newcommand{\calO}{\mathcal{O}}
\newcommand{\real}{\mathbb{R}}
\newcommand{\T}{\intercal}

\DeclareMathOperator*{\Exp}{Exp}

\DeclareMathOperator*{\vect}{Vec}
\DeclareMathOperator*{\diag}{Diag}

\newcommand{\inner}[2]{\langle #1 \,, #2 \rangle}

\begin{document}
\maketitle

\begin{abstract}
Invex programs are a special kind of non-convex problems which attain global minima at every stationary point. While classical first-order gradient descent methods can solve them, they converge very slowly. In this paper, we propose new first-order algorithms to solve the general class of invex problems. We identify sufficient conditions for convergence of our algorithms and provide rates of convergence.  Furthermore, we go beyond unconstrained problems and provide a novel projected gradient method for constrained invex programs with convergence rate guarantees. We compare and contrast our results with existing first-order algorithms for a variety of unconstrained and constrained invex problems. To the best of our knowledge, our proposed algorithm is the first algorithm to solve constrained invex programs.
\end{abstract}
	
\section{Introduction}
\label{sec:introduction}

Many learning problems are modeled as optimization problems. With the explosion in deep learning, many of these problems are modeled as non-convex optimization problems --- either by using non-convex objective functions or by the addition of non-convex constraints. While well-studied algorithms with fast convergence guarantees are available for convex problems, such mathematical tools are more limited for non-convex problems. In fact, the general class of non-convex optimization problems is known to be NP-hard~\citep {jain2017non}. Coming up with global certificates of optimality is the major difficulty in solving non-convex problems. In this paper, we take the first steps towards solving a special class of non-convex problems, called \emph{invex problems}, which attain global minima at every stationary point~\citep{hanson1981sufficiency,ben1986invexity}. Invex problems are tractable in the sense that we can use local certificates of optimality to establish the global optimality conditions. 
	
\paragraph{Related work.}

\begin{figure}[!ht]
    \centering
    \begin{subfigure}{0.4\textwidth}
	\begin{center}
	\begin{tikzpicture}[scale=0.6]
		\draw[fill=blue!10] (3, 3.2) circle (2.7cm);
		\draw[] (0, 0) rectangle (6, 6);
		\draw[fill=white] (3, 3.2) circle (1.7cm);
		\draw[fill=white] (3, 3.2) circle (1cm);
		\node[] at (3, 1) {\small{Invex}};
		\node[] at (3, 2) {\small{G-convex}};  
		\node[] at (3, 3) {\small{Convex}};     
	\end{tikzpicture}
	\caption{\label{fig:hierarchy}Hierarchy of different function classes}
	\end{center}
    \end{subfigure}%
    \begin{subfigure}{0.3\textwidth}
        \includegraphics[scale=0.25]{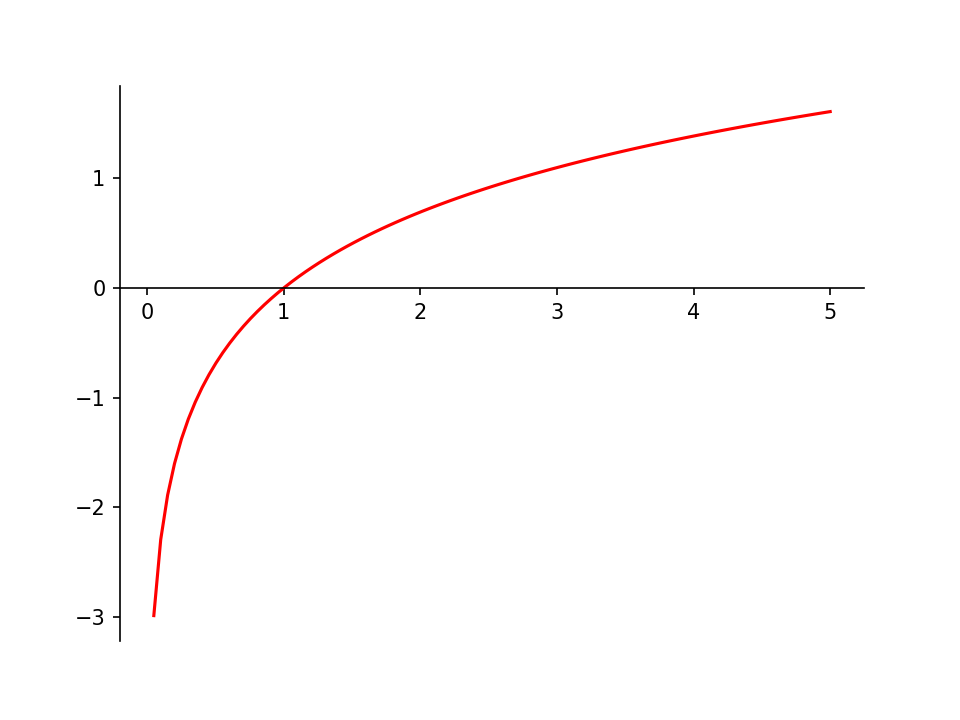}
        \caption{\label{fig:g-convex}G-convex function }
    \end{subfigure}%
    \begin{subfigure}{0.3\textwidth}
        \includegraphics[scale=0.25]{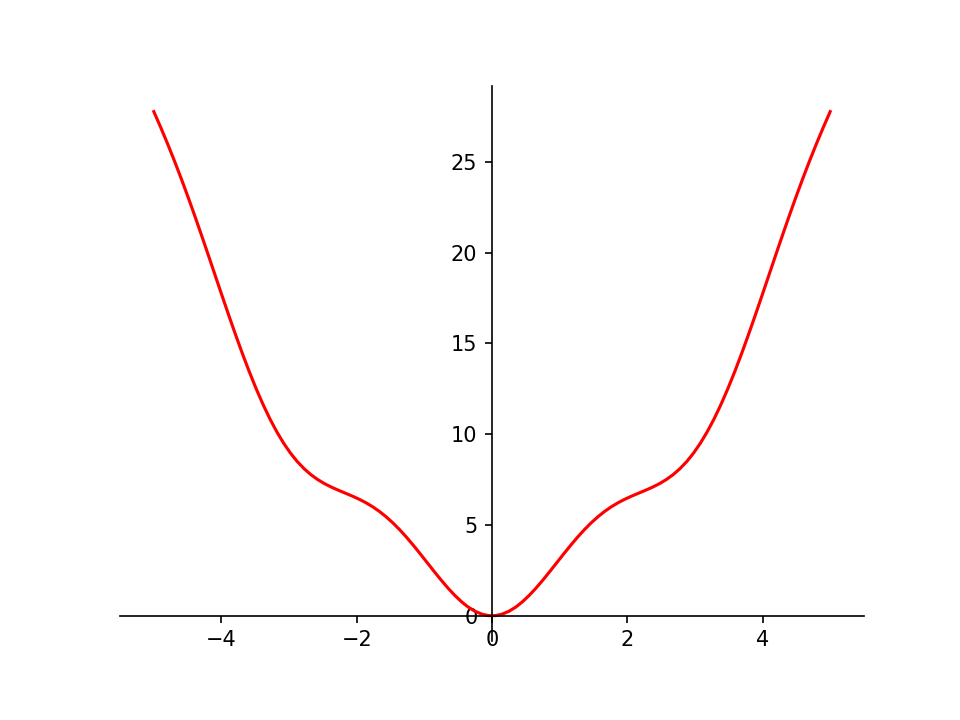}
        \caption{\label{fig:invex}Invex function }
    \end{subfigure}
    \caption{\label{fig:illustration} Figure~\ref{fig:hierarchy} shows hierarchy of different function classes, Figures~\ref{fig:g-convex} and~\ref{fig:invex} show an example of geodesically convex (G-convex) function $f(x) = \log x$ for $x > 0$ and invex function $f(x) = x^2 + 3 \sin^2 x$ respectively -- both of them are non-convex.}
\end{figure}
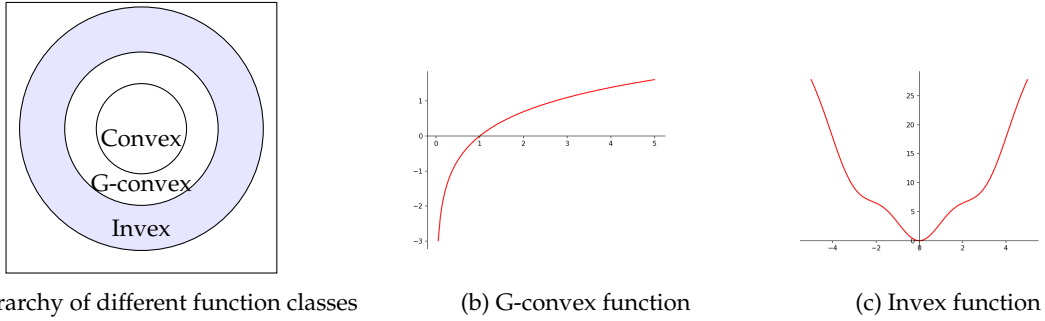 
First-order gradient descent methods are the most well-known algorithms to solve convex optimization problems. While they can also solve invex optimization problems under certain conditions, they can be really slow in their convergence due to their inability to use any underlying `invex' geometry. \cite{zhang2016first} have studied the minimization of a special class of unconstrained invex functions -- called geodesically convex functions. They provide convergence rate guarantees for their algorithms assuming upper bounds on the sectional curvature of the manifold. Such algorithms have also been studied by \cite{udriste2013convex} in their work on optimization methods on Riemannian manifolds, albeit with a focus on asymptotic convergence. The simplest instance of geodesically convex optimization is more commonly known as geometric programming~\citep{boyd2007tutorial,udriste2013convex}. The algorithms solving geodesically convex problems use topological properties of geodesic curves for their convergence. Often, finding the underlying geodesic curves and characterizing the manifold prove to be the bottleneck for solving such problems. These difficulties extend naturally to the general class of invex problems where topological properties are difficult to establish. In this work, while we do connect properties of invex functions with the topology of the domain, we also develop algebraic methods for implementing our proposed algorithm. Our focus in this work is to develop first-order methods with provable global convergence rates for a broader class of invex problems. Our method reduces to classical gradient descent (Riemannian gradient descent~\cite{zhang2016first,boumal2020introduction}) if the underlying function is convex (geodesically convex). Many optimization problems can be classified as invex problems by using the simple characterization by \cite{ben1986invexity}. We provide some such examples that have been studied in recent years as motivation. \cite{pini1994convexity} showed that geodesically convex functions are invex. This means that problems such as matrix Karcher mean problem~\citep{zhang2016first}, power control~\citep{boyd2007tutorial}, optimal doping profile~\citep{boyd2007tutorial} and non-convex matrix factorization~\citep{luo2022nonconvex} are invex. Any function which satisfies PL-inequality~\citep{karimi2016linear} is an invex function. This implies that averaged-out non-convex functions~\citep{wang2022provable} are also invex. Similarly, quasar-convex functions~\citep{hinder2020near} can also be shown to be invex. Recent studies have shown that many machine learning problems such as learning output kernels~\citep{dinuzzo2011learning}, multiple tasks learning~\citep{ciliberto2015convex}, minimum distance lasso~\citep{lozano2016minimum}, reinforcement learning with general utilities~\citep{zhang2020variational},  fair sparse regression~\citep{barik2021fair}, sparse mixed linear regression~\citep{barik2022sparse}, imaging with invex regularizers~\citep{pinilla2022improved} and DAG learning~\citep{bello2022dagma} are also invex. Identifying a problem to be invex is relatively a simple task, whereas coming up with an efficient algorithm to solve such a problem by leveraging the invexity is quite challenging. Furthermore, convergence rate analysis of such algorithms becomes even more tedious. To the best of our knowledge, we are not aware of any provably convergent general algorithm to solve invex problems. 
In this paper, we present first-order methods to solve invex problems with provable convergence rate guarantees under some natural technical conditions. 
		
\paragraph{Summary of contributions}
\begin{enumerate}
\setlength{\itemsep}{0pt}
    \item We present a first-order gradient descent algorithm for invex problems (Algorithm~\ref{alg:invex grad alg}). We demonstrate the feasibility of our update rule over a wide variety of examples (Section~\ref{subsec: invex gradient descent method}).
    \item As an extension, we propose a projected gradient descent algorithm for constrained invex problems (Algorithm~\ref{alg:proj invex grad alg}). We show that our algorithm works for constrained geodesically convex programs in Hadamard manifolds (Section~\ref{subsubsec: constrained geodesically convex problem}).     
    \item We provide convergence rate guarantees (Table~\ref{tab:convergenc rate table}) for our proposed algorithms (Theorem~\ref{thm:convergence of L smooth functions}, \ref{thm:convergence of invex f with triangle inequality}, \ref{thm:convergence of strongly invex functions}, \ref{thm:convergence of invex f with triangle inequality projected}, \ref{thm:convergence of strongly invex functions projected}) under varying degree of assumptions. We identify sufficient technical conditions needed for the convergence of our proposed algorithm. 
    \item Finally, we show the applicability of our algorithms on both unconstrained~\citep{bello2022dagma} and constrained~\citep{barik2021fair,barik2022sparse} machine learning problems in Section~\ref{sec:Applications}. We show that under the same initialization and hyperparameters, our algorithms outperform the standard gradient descent algorithms.  
\end{enumerate}
	
\begin{table}[!ht]
  \caption{Convergence rate guarantees of first-order methods under different settings after $k$ iterations and $\delta \in (0, 1]$. For $L$-smooth convex/invex functions convergence is studied as $f(x_k) \to f(x^*)$ while for strongly convex/invex function it is studied as $x_k \to x^*$ (or equivalent). }
  \label{tab:convergenc rate table}
  \centering
  \begin{tabular}{lcccc}
    \toprule
         & Convex & G-convex & GD on Invex & Our method on Invex  \\
    \midrule
    \textbf{Unconstrained} & & & & \\
    $L$-smooth convex/invex     & $\calO(\frac{1}{k})$ & $\calO(\frac{1}{\sqrt{k}})$ & $\times$ &  $\calO(\frac{1}{k})$   \\
    Strongly convex/invex     &  $\calO((1 - \delta)^k)$      & $\calO((1 - \delta)^k)$  & $\times$ & $\calO((1 - \delta)^k)$  \\
    \textbf{Constrained} & &  & & \\
    $L$-smooth convex/invex     & $\calO(\frac{1}{k})$  & $\times$ & $\times$ & $\calO(\frac{1}{k})$     \\
    Strongly convex/invex     & $\calO((1 - \delta)^k)$        & $\times$   & $\times$ & $\calO((1 - \delta)^k)$  \\
    \bottomrule
  \end{tabular}
\end{table}	
\section{Invexity}
\label{sec:invexity and convexity along the curve}
	
In this section, we formally define the invex function and relate it with convexity along the curve. Consider a differentiable function $\phi(x)$ defined on a Riemannian manifold $\calC$. Let $\inner{\cdot}{\cdot}_x$ be the inner product in the tangent space $T_x\calC$ of $x$ induced by the Riemannian metric. 
	
\begin{definition}[Invex functions]
\label{def:invex functions}
    Let $\phi(x)$ be a differentiable function defined on $\calC$. Let $\eta$ be a vector valued function defined in $\calC \times \calC$ such that $\inner{\eta(y, x)}{\nabla \phi(x)}_x$ is well-defined $\forall x, y \in \calC$. Then, $\phi(x)$ is an $\eta$-invex function if
    \begin{align}
	\label{eq:invexity}
	\phi(y) - \phi(x) \geq \inner{\eta(y, x)}{\nabla \phi(x)}_x, \forall x, y \in \calC \; .
    \end{align}
\end{definition}
If the manifold $\calC$ is $\real^n$, then we get the standard definition of invex functions~\citep{hanson1981sufficiency}. Convex functions are invex functions on $\real^n$ with $\eta(y, x) = y - x$. In that sense, invex functions are a generalization of convex functions. \cite{ben1986invexity} proved the sufficient and necessary condition that any stationary point of the invex function is the global minima. It follows that (at least in $\real^n$) any algorithm that converges to a stationary point, in principle, can solve unconstrained invex problems. However, convergence rate guarantees are not available for any such algorithms.  Similarly, geodesically convex functions on the Riemannian manifold $\calC$ are $\eta$-invex with $\eta(y, x) = \Exp^{-1}_x(y)$ where $\Exp^{-1}_x$ is the inverse of the exponential map $y = \Exp_x(v)$ for some $v$ in the tangent space at point $x \in \calC$. This motivates to characterize invex functions by treating them as convex functions along a curve. More formally, we provide the following proposition from \cite{pini1994convexity}.
	
\begin{proposition}[Proposition 4.1 from \cite{pini1994convexity}]
\label{prop:invexity curve}
    A differentiable real function $\phi(x)$ defined on $\calC$ is $\eta$-invex if and only if for every $x, y \in \calC$, the real function $g_{x, y}(t) = f(\gamma_{x, y}(t))$ is convex on $[0, 1]$ for some curve $\gamma_{x,y}$ such that 
    \begin{align}
    \label{eq:itegrable gamma}
    \begin{split}
	&\gamma_{x, y}(0) = x,\; \gamma_{x, y}(1) = y, \;\dot{\gamma}_{x, y}(u) (t - u) = \eta(\gamma_{x, y}(t), \gamma_{x, y}(u)),\; \forall t, u \in [0, 1] \; .
    \end{split}
    \end{align}  
\end{proposition} 

Proposition~\ref{prop:invexity curve} immediately provides a setting for $\eta(y, x)$ in terms of underlying curve, i.e., $\eta(y, x) = \dot{\gamma}_{x, y}(0)$. 
For convex functions, the underlying curve $\gamma_{x, y}(t) = x + t (y - x)$. Similarly, for a geodesically convex function, the underlying curve $\gamma_{x, y}(t)$ is the geodesic curve joining $x$ and $y$. We notice, however, that finding the underlying curve for any given $\eta$-invex function may not be an easy task. We observe that proposition \ref{prop:invexity curve} allows us to connect invexity of a function to a geometric property (underlying curves) of the domain of the function.
This leads us to define invex sets as a natural extension of convex sets.
	
\begin{definition}[Invex set]
    \label{def:invex set}
    A set $\calA \subseteq \calC$ is called $\eta$-invex set if $\calA$ contains every curve $\gamma_{x, y}$  of $\calC$ as defined in proposition~\ref{prop:invexity curve} whose endpoints $x$ and $y$ are in $\calA$.
\end{definition}
	
It is also possible to characterize invex sets using $\eta(y, x)$ functions by using the relationship between $\gamma_{x, y}$ and $\eta(y, x)$ from equation~\eqref{eq:itegrable gamma}. Thus, we sometimes refer to the invex set as $\eta$-invex set with the assumption that $\eta(y, x)$ is computed using $\gamma_{x, y}$. 
We note that our definition is a slight departure from the definition of the invex set used in \cite{mohan1995invex,noor2005invex}. However, we find our definition more natural for our purpose. 
Our optimization problem involves minimizing an $\eta$-invex function over an $\eta$-invex set. In the remaining paper, we would assume $\calC$ to be an $\eta$-invex set unless stated otherwise.
	
\begin{definition}[Invex program]
    \label{def:invex program}
    Let $f: \calC \to \real$ be an $\eta_1$-invex function, and $g_i: \calC \to \real, \forall i \in \{ 1, \cdots, m \}$ be $\eta_2$-invex functions, then the optimization problem
    \begin{align}
        \label{eq:invex program}
	\begin{split}
	    \begin{matrix}
			\min_{x \in \calC} & f(x), & 
			\text{such that} & g_i(x) \leq 0 ,& \forall i \in  \{ 1, \cdots, m \}
		\end{matrix}
	\end{split}
    \end{align} 
is called an invex program. 
\end{definition} 

It is possible to include equality constraints in the program, but we opt for only inequality constraints for simplicity. 
%
Invex programs without any constraints are called unconstrained invex programs. In the next section, we propose a first-order method to solve invex programs.
	
\section{New first order algorithm}
\label{sec:new first order algorithm}
	
In this section, we develop first-order gradient descent methods for invex programs. We start with the unconstrained version of problem~\eqref{eq:invex program} and then gradually build up our method for the constrained version.

\subsection{Invex gradient descent method}
\label{subsec: invex gradient descent method}
	
The main task in our algorithm is to figure out a $y \in \calC$ for a given $x \in \calC$ and a direction $v \in T_x\calC$ such that $\eta(y, x) = v$. Such a $y$ need not be unique and we are only interested in finding one $y$ (of possibly many) that satisfies $\eta(y, x) = v$. We provide the following gradient descent algorithms to solve invex programs. 

\begin{minipage}[t]{0.45\textwidth}		
\begin{algorithm}[H]
    \caption{Invex Gradient Descent}\label{alg:invex grad alg}
    \begin{algorithmic}
        \STATE \textbf{Input: } $x_0$
        \FOR{$k\leq T$} 
        \STATE \textbf{Update: }Find an $x_{k+1}$ such that $ \eta(x_{k+1}, x_k) = -\alpha_k \nabla f(x_k)$
        \ENDFOR 
    \end{algorithmic}
\end{algorithm}
\vspace{\baselineskip}
\end{minipage}
\hfill
\begin{minipage}[t]{0.5\textwidth}
\begin{algorithm}[H]
    \caption{Projected Invex Gradient Descent}\label{alg:proj invex grad alg}
    \begin{algorithmic}
        \STATE \textbf{Input: } $x_0$
        \FOR{$k\leq T$} 
        \STATE \textbf{Update: } Find a $y_{k+1}$ such that $ \eta_1(y_{k+1}, x_k) = -\alpha_k \nabla f(x_k)$
        \STATE \textbf{Projection step: } $x_{k+1} \gets \rho_{\eta_2}(y_{k+1})$
        \ENDFOR 
    \end{algorithmic}
\end{algorithm}    
\end{minipage}

In Algorithm~\ref{alg:invex grad alg}, $T$ is the maximum number of iterations and $\alpha_k$ is the step size which depends upon the particular implementation of the algorithm. We will specify a particular choice of $\alpha_k$ in the convergence rate analysis of Algorithm~\ref{alg:invex grad alg}. Without any information on underlying curve $\gamma_{x, y}(t)$, the update step of Algorithm~\ref{alg:invex grad alg}, i.e., finding a $y \in \calC$ such that $\eta(y, x) = v$ is a problem-dependent task. Below we provide an array of examples to explain this observation.
		
\begin{example}[Convex case]
    For convex problems, $y = x + v$.
\end{example}
\begin{example}[Geodesically convex case]
    For geodesically convex problems, $y = \Exp_x(v)$ where $\Exp$ is exponential map as defined in~\cite{zhang2016first}. 
\end{example}

\begin{example}[PL inequality]
    It is known that the functions satisfying PL-inequality are invex~\citep{karimi2016linear}. However, this characterization does not readily lead to a good $\eta(y, x)$. We provide an $\eta$ function in the following lemma which can be used in the update step. 
    \begin{lemma}
        \label{lem:pl inequality}
        Let $f(x)$ be an $L$-smooth function that satisfies PL inequality for some $\mu > 0$. Then it is $\eta$-invex for $\eta(y, x) = \frac{1}{\mu}\left(\nabla f(y) + L \frac{\| y - x \|}{\| \nabla f(x) \|}   \nabla f(x) \right) $.
    \end{lemma}
    The proof of Lemma~\ref{lem:pl inequality} and further discussion is deferred to Appendix~\ref{sec:functions satisfying pl inequality}. 
\end{example}

\begin{example}[Quasar Convex Functions]
    \cite{hinder2020near} showed that quasar convexity implies invexity. However, they do not provide any $\eta$ for the quasar convex functions. In the following lemma, we provide an $\eta$ for quasar convex functions.
    \begin{lemma}
        \label{lem:quasar convex functions}
        For any $\nu \geq 0$, there exists a $\beta \in [0, 1]$ such that quasar convex functions are $\eta$-invex for $\eta(y, x) = \frac{\beta}{\nu(1 - \beta)} (y - x) $. 
    \end{lemma}
    This leads to the update $ y = x + \nu \frac{1 - \beta}{\beta} v$. We provide the proof of Lemma~\ref{lem:quasar convex functions} in Appendix~\ref{sec:quasar convex functions}. 
\end{example}

\begin{example}[Connection with Bregman divergence and Mirror descent]
    Let $B_{\psi}(y, x)$ be the Bregman divergence associated with a strongly convex and differentiable function $\psi(x)$ such that $B_{\psi}(y, x) = \psi(y) - \psi(x) - \inner{\nabla \psi(x)}{y - x}$.
    Let $\eta(y, x) = \nabla B_{\psi}(y, x)  = \nabla \psi(y) - \nabla \psi(x)$, i.e., $\eta(y, x)$ is a conservative field and $B_{\psi}(y, x)$ is its potential function. Then a typical mirror descent update~\citep{duchi2010composite} can be used to compute $y$, i.e., $y = \inf_u B_{\psi}(u, x) + \alpha \inner{\nabla f(x)}{u - x}$.
\end{example}
  
\begin{example}[Recent Invex Problems]
    Some recently studied problems in invexity such as \cite{barik2021fair} and \cite{barik2022sparse} are invex for a particular form of $\eta(y, x)$. In particular, consider any point $x \in \real^n$  of the form $x = \begin{bmatrix} x_1 & x_2
        \end{bmatrix}^\T$   where $x_1 \in \real^{n_1}, x_2 \in \real^{n_2}$ such that $n_1 + n_2 = n$. Then for any two $x, y \in \real^n$, $\eta(y, x)$ takes the form $\eta(y, x) = \begin{bmatrix}
            y_1 - x_1 &
            A(y_1, x_1) (y_2 - x_2)
        \end{bmatrix}^\T$
    where $A(y_1, x_1) \in \real^{n_2 \times n_2}$ and $A(y_1, x_1) \succ \mathbf{0}, \forall y_1, x_1 \in \real^{n_1}$. For such problems, update step in Algorithm~\ref{alg:invex grad alg} becomes $y_1 = x_1 + v, y_2 = x_2 + A(y_1, x_1)^{-1} v$.
\end{example}

\begin{example}[A generic approach using function inverse]
    A generic approach to compute $y$ such that $\eta(y, x) = v$ is to treat $\eta(y, x) = g(y)$ for a fixed $x$ and then compute $y = g^{-1}(v)$. This approach works as long as we have explicit closed-form expression for $g^{-1}(v)$. For our purpose, we ignore the uniqueness of $y = g^{-1}(v)$ and allow any $y$ as long as $g(y) = v$.  
\end{example}

\subsection{Convergence of invex gradient descent method}
\label{subsec:converegence of invex gradient descent method}
	
We start the convergence analysis of Algorithm~\ref{alg:invex grad alg} with the weakest set of assumptions, and then we gradually add stronger conditions to get a better convergence rate. Before we delve into our first result of convergence, we define a notion of smoothness in the invex manifold $\calC$.
	
\begin{definition}[L-smooth function]
    A differentiable function $f:\calC \to \real$ is called $L$-smooth on an $\eta$-invex set $\calC$ if $\forall x, y \in \calC$
    \begin{align}
	\label{eq:L-smooth}
	f(y) \leq f(x) + \inner{\eta(y, x)}{\nabla f(x)}_x + \frac{L}{2} \| \eta(y, x) \|^2, 
    \end{align}
    where norm $\| \cdot \|$ is induced by the Riemannian metric at $x$.
\end{definition} 
	
Note that a function $f$ need not be an invex function to be an $L$-smooth function. Our first convergence guarantee is for $L$-smooth functions.
	
\begin{theorem}[Convergence of $L$-smooth functions]
    \label{thm:convergence of L smooth functions}
    Let $f$ be a $L$-smooth function and $f^* = \min_{x \in \calC} f(x) \geq B$ for some $B > - \infty$. If $x_k$ is computed using Algorithm~\ref{alg:invex grad alg} with $\alpha_k = \alpha \in (0, \frac{2}{L})$, then $\lim_{k \to \infty} \| \nabla f(x_k) \| = 0$.
\end{theorem} 

Theorem~\ref{thm:convergence of L smooth functions} states that Algorithm~\ref{alg:invex grad alg} converges to a stationary point even if the function is not invex. Our next task is to achieve a better convergence rate by adding the assumption that $f$ is an invex function. However, to do that, we need to impose an extra condition on the choice of $\eta(\cdot, \cdot)$ which in turn imposes an extra condition on the geometry of $\calC$. To make Algorithm~\ref{alg:invex grad alg} amenable to rigorous convergence rate analysis, we impose a sufficient condition on the geometry of $\calC$ which is analogous to triangle inequality in Euclidean space.

\begin{assumption}[Triangle Inequality]
    \label{assum:triangle inequality}
    Let $x, y, z \in \calC$, then for some $b, c > 0$ 
    \begin{align}
        \label{eq:triangle inequality}
	\begin{split}
		\| \eta(y, z) \|^2 \leq & \| \eta(x, z) \|^2 + b \| \eta(y, x) \|^2 -  c \inner{ \eta(y, x) }{ \eta(z, x)}_x \; .
	\end{split}
    \end{align}
\end{assumption}
	 
The triangle inequality assumption is an assumption on the geometry of manifold $\calC$. We also note that Euclidean spaces clearly satisfy Assumption~\ref{assum:triangle inequality} by simply taking $b=1$ and $c = 2$. \cite{zhang2016first} showed that any Riemannian manifold with sectional curvature upper bounded by $\kappa \leq 0$ also satisfies Assumption~\ref{assum:triangle inequality}. Now, we are ready to state our second convergence result.
			
\begin{theorem}[Convergence of invex functions]
    \label{thm:convergence of invex f with triangle inequality}
    Let $f: \calC \to \real$ be an $L$-smooth $\eta$-invex function such that $\calC$ satisfies Assumption~\ref{assum:triangle inequality}. Furthermore, let $x^* = \arg\min_{x \in \calC} f(x)$ such that $f(x^*) > -\infty$ and $\| \eta(x_0, x^*) \| \leq M < \infty$. If $x_k$ is computed using Algorithm~\ref{alg:invex grad alg} with $\alpha_k = \alpha \in (0,  \frac{2}{L})$, then $f(x_k)$ converges to $f(x^*)$ at the rate $\calO(\frac{1}{k})$.  
\end{theorem}
			
We further improve convergence rate results by imposing even more conditions on function $f$. We define $\mu$-strongly $\eta$-invex functions as a natural extension to $\mu$-strongly convex functions as follows.
			
\begin{definition}[$\mu$-strongly $\eta$-invex function]
    \label{def:strongly invex functions}
    A differentiable function $f:\calC \to \real$ is called $\mu$-strongly $\eta$-invex function for some $\mu > 0$ if $f(y) \geq f(x) + \inner{\eta(y, x)}{\nabla f(x)}_x + \frac{\mu}{2} \| \eta(y, x) \|^2$, where norm $\| \cdot \|$ is induced by the Riemannian metric at $x$.
\end{definition}
We provide the following convergence results for the $\mu$-strongly $\eta$-invex functions.
			
\begin{theorem}[Convergence of strongly invex functions]
    \label{thm:convergence of strongly invex functions}
    Let $f:\calC \to \real$ be an $L$-smooth $\mu$-strongly $\eta$-invex function such that $\calC$ satisfies Assumption~\ref{assum:triangle inequality}. Furthermore, let $x^* = \arg\min_{x \in \calC} f(x)$ such that $f(x^*) > -\infty,  \| \eta(x_0, x^*) \| \leq M < \infty$ and $\| \eta(y, x) \|_x^2 \leq R \| \eta(x, y) \|_y^2, \forall x, y \in \calC$ for some $R > 0$. If $x_k$ is computed using Algorithm~\ref{alg:invex grad alg} with $\alpha_k = \alpha \in (0,  \min (\frac{2}{ R \mu c }, \frac{c}{2bL}))$, then 
    \begin{align}
	\label{eq:convergence of strongly invex}
	\| \eta(x_{k+1} , x^*) \|^2 \leq \left(1 - \frac{c \alpha R \mu}{2}\right)^{k+1} M^2 \; .
    \end{align}
\end{theorem}
			
We have intentionally chosen to show convergence results for a constant step size of $\alpha_t$ for simplicity. It is not difficult to get better convergence rates by carefully choosing $\alpha_t$. It is easy to verify that all our results hold for the convex case. They also extend nicely to all the results in \cite{zhang2016first} for geodesically convex case. 
	 
\subsection{Projected invex gradient descent method}
\label{subsec: projected invex gradient descent method}
			
Now that we have shown convergence results for unconstrained invex programs. We can extend these results to constrained case by providing a projected invex gradient descent method. We first discuss projection on an invex set before providing the algorithm. %
Let $\calA \subseteq \calC$ be an $\eta$-invex set. We define the projection of $x \in \calC$ on $\calA$ as a retraction.
			
\begin{definition}[Projection on invex set]
    \label{def: Projection on invex set}
    Let $\gamma_{x, y}(t)$ be the curve connecting $x \in \calC$ to $y \in \calA$ such that $\gamma_{x, y}(0) = x$ and $\gamma_{x, y}(1) = y$. Projection $\rho_{\eta}(x)$ of $x$ on $\calA$ is defined as $\rho_{\eta}(x) = \arg\min_{ y \in \calA} \| \eta(y, x) \|$.
\end{definition}
			
It is easy to see that for convex sets, projection reduces to finding $y \in \calA$ which is closest to $x$ in Euclidean distance. Also, notice that if $x \in \calA$, then $\rho_{\eta}(x) = x$. 
First, observe that in the invex program as defined in~\ref{eq:invex program} objective function $f$ is $\eta_1$-invex while the constraint set is $\eta_2$-invex. Thus, we make update in two steps. The first step works in $\eta_1$-invex set and then it is projected back to $\eta_2$-invex set. The convergence rates of the invex gradient descent algorithm can be extended to the projected invex gradient descent algorithm under the following condition (Details in Appendix~\ref{subsec:convergence of projected invex gradient method}).

\begin{assumption}[Contraction]
    \label{assum:contraction}
    Let $x, y \in \calC$ and $\rho_{\eta_2}(x), \rho_{\eta_2}(y)$ are their projection on an $\eta_2$-invex set $\calA$ respectively. Then, $\| \eta_1( \rho_{\eta_2}(y), \rho_{\eta_2}(x) ) \|_{\rho_{\eta_2}(x)} \leq  \| \eta_1( y, x ) \|_{x} $.
\end{assumption}

Assumption~\ref{alg:proj invex grad alg} clearly holds for convex objective functions on convex constraints and thus, it is a natural choice of assumption to impose on the general case of constrained invex programs. In fact, in the next subsection, we show that it also holds for geodesically convex programs.
			
\subsection{Constrained geodesically convex problem}
\label{subsubsec: constrained geodesically convex problem}

In recent literature, there has been a lot of focus on constrained geodesically convex problems (with both the objective function and constraints being geodesically convex)~\citep{weber2022riemannian,liu2020simple,zhang2022minimax}. Our projected gradient algorithm~\ref{alg:proj invex grad alg} works for constrained case of geodesically convex optimization problems with sectional curvature upper bounded by $\kappa \leq 0$. To that end, we can show that Assumption~\ref{assum:contraction} holds in this particular case.
			
\begin{theorem}[Projection contracts for geodesically convex sets with negative curvature]
    \label{thm: contraction geodesically convex}
    Let $\rho$ be the projection operator as defined in~\ref{def: Projection on invex set} on a closed geodesically convex subset $\calA$ of a simply connected Riemannian manifold $\calC$ with sectional curvature upper bounded by $\kappa \leq 0$. Then the projection satisfies Assumption~\ref{assum:contraction}.
\end{theorem}

Thus, we can use Algorithm~\ref{alg:proj invex grad alg} to solve constrained geodesically convex problems with sectional curvature upper bounded by $\kappa \leq 0$. This extends all the results from~\cite{zhang2016first} to the constrained case and provides a novel method to solve constrained geodesically convex problems.

\section{Applications}
\label{sec:Applications}
				
In this section, we provide specific examples of invex programs to validate our theory. Our task is to provide a working $\eta(y, x)$ for all the problems and explicitly construct the update step and projection step (if needed). Finally, we compare the performance of our algorithm with the gradient descent (or projected gradient descent) algorithm. The latter of which provides no convergence rate guarantees for invex problems. We chose to go with the vanilla implementation for both algorithms, i.e., without any performance enhancement tricks such as line search. This was done to ensure that the comparison remains fair as our algorithms can also be adapted to include such tricks to further boost its performance. However, that is not our focus in this work.
				
\subsection{Log-determinant acyclicity characterization for DAGs}
\label{subsec: acyclicity in dags}
				
We start with an unconstrained invex program. \cite{bello2022dagma} provided a novel characterization of the acyclicity of DAGs in their recent work. Their characterization employs a log-determinant function and is stated in Theorem 1~\citep{bello2022dagma}. Let $ \mathcal{W} = \triangleq \{ W \in \mathbb{R}^{d \times d} \mid s > r(W \circ W)  \}$. Their log-determinant acyclicity characterization of DAGs uses the function $h(W) = - \log \det (sI - W \circ W) + d \log s$
where $W \in \mathcal{W}$, $I$ is identity matrix, $r(\cdot)$ denotes the spectral radius and $A \circ B$ denotes Hadamard product between two matrices. We take $s$ to be $1$ without loss of generality and thus $h(W) = - \log \det (I - W \circ W)$. They show in Corollary 3~\citep{bello2022dagma} that $h(W)$ is an invex function. However, they do not provide any specific $\eta$ for invexity. Next, we will provide a possible $\eta$ for the problem but before that, we need to define Hadamard division of two same-sized matrices $A, B \in \real^{d \times d}$ as $(A \oslash B)_{ij} = \frac{A_{ij}}{B_{ij}}$ when $B_{ij} \ne 0$ and $0$ otherwise. Now we are ready to state the following lemma.
\begin{lemma}
    \label{lem: eta for h(W)}
    The function $h(W) = - \log \det (I - W \circ W), \forall W \in \mathcal{W}$ is $\eta$-invex for $\eta(U, W) = -\frac{1}{2} ((I - W \circ W) (\log(I - U \circ U) - \log(I - W \circ W))) \oslash W$.
\end{lemma}

We can use our proposed $\eta$ to construct updates for Algorithm~\ref{alg:invex grad alg}. Observe that for an stepsize $\alpha$, update step in Algorithm~\ref{alg:invex grad alg} is $ \eta(W_{k+1}, W_k) = - \alpha \nabla h(W_k) $. We take $M = (I - W_{k+1} \circ W_{k+1})$ and $N = (I - W_k \circ W_k)$ for clarity. Then the update step becomes $	-\frac{1}{2} (N (\log M - \log N)) \oslash W_k = - 2 \alpha  N^{-\T} \circ W_k$ and we get $M = \exp( \log N + 4 \alpha N^{-1} ((N^{-\T} \circ W_k) \circ W_k ))$ which provides the update step. 
\begin{figure}[!ht]
	\centering
	\begin{subfigure}{.45\textwidth}
		\centering
		\includegraphics[width=\linewidth]{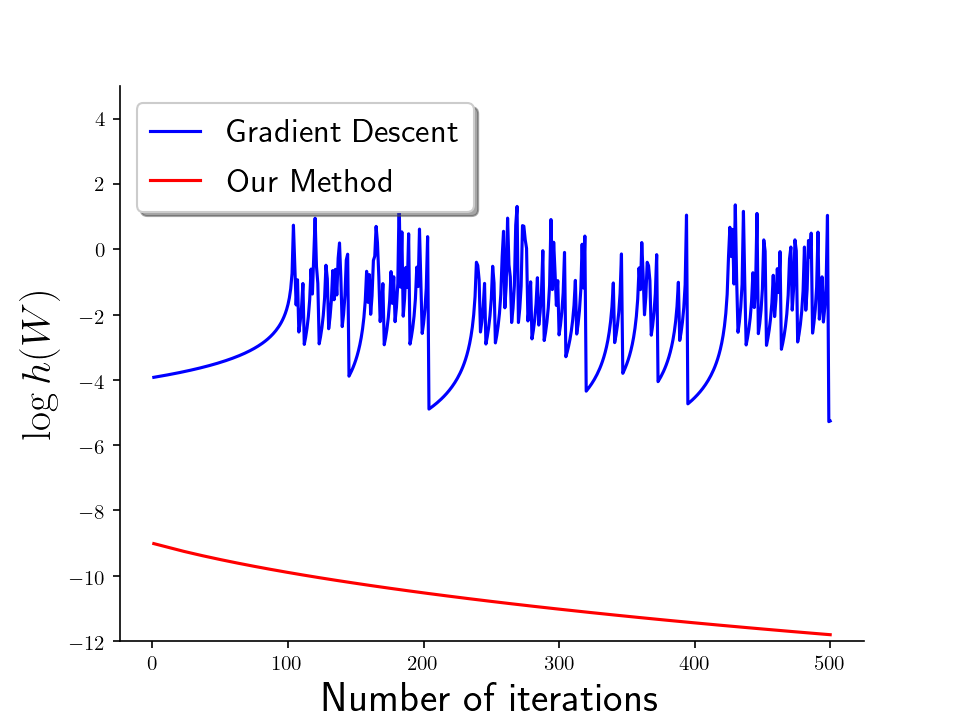}
		\caption{$n=20$}
	\end{subfigure}%
	\begin{subfigure}{.45\textwidth}
		\centering
		\includegraphics[width=\linewidth]{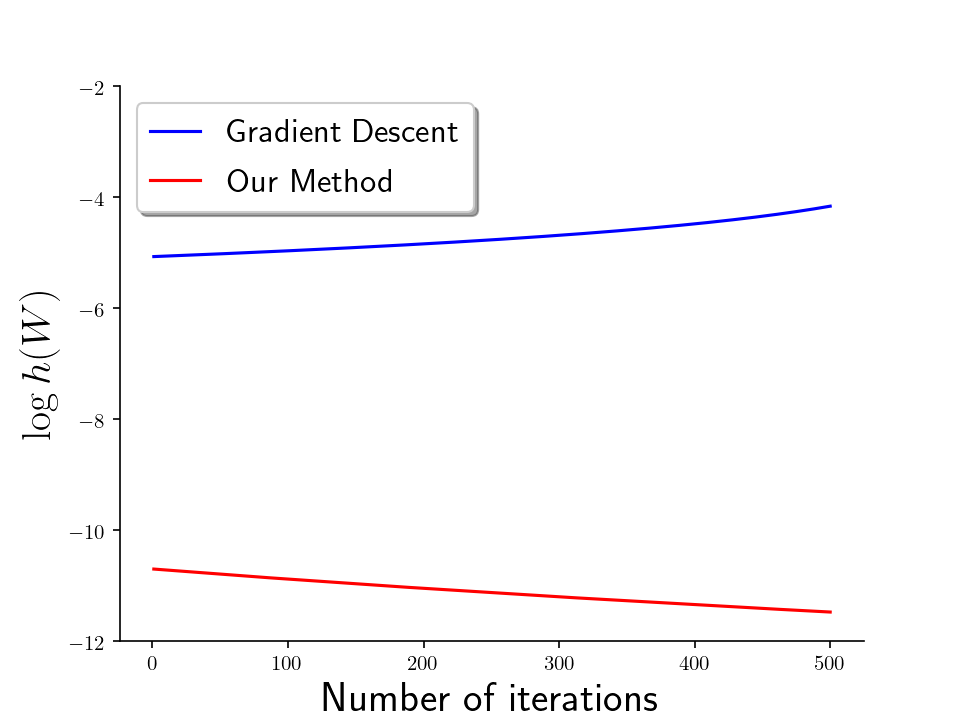}	
		\caption{$n=50$}
	\end{subfigure}%
	\caption{\label{fig:dag} $\log h(W)$ vs number of iterations - comparing Algorithm~\ref{alg:invex grad alg} with standard gradient descent algorithm for graphs on $20$ (left) and $50$ (right) nodes.}
\end{figure}  
We used this update step to implement Algorithm~\ref{alg:invex grad alg}. The performance of our algorithm was compared against the standard gradient descent algorithm. Both the algorithms were run with a random initialization of $W$ which was kept the same for both algorithms. We found that the gradient descent algorithm failed to converge in several instances but our algorithm converged towards zero objective function value as predicted by \cite{bello2022dagma} (See Figure~\ref{fig:dag}). 

\subsection{Fair sparse regression}
\label{subsec:fair sparse regression}

Our next example is a constrained invex program. \cite{barik2021fair} proposed a novel invex relaxation for fair sparse regression problem. In this problem, each data point is associated with one of the two groups, and the response variable is generated with a signed bias term based on the membership to the group. They use the following generative model: $ 	y_i = X_i^\T w^* + \gamma z_i^* + e_i, \; \forall i \in \{1, \cdots, n\}$,
where $e_i$ is a zero mean independent additive noise and $z_i^*$ is the group membership. The task is to identify regression vector $w^* \in \real^d$ along with $z_i^*$ for every data point. \cite{barik2021fair} proposed the following invex relaxation for this problem. 
\begin{align}
    \label{eq: invex fair lasso}
    \begin{matrix}
        \min_{w, Z} & \inner{M(w)}{Z} + \lambda_n \| w \|_1 & \text{such that} & \diag(Z) = 1, \; Z \succeq 0 
    \end{matrix}\; ,
\end{align}
where 
\begin{align}
    M(w) \triangleq \begin{bmatrix}
        \frac{1}{n}(Xw - y)^\T (Xw - y) + 1 & \frac{\gamma}{n}(Xw - y)^\T \\
        \frac{\gamma}{n}(Xw - y) & (\frac{\gamma^2}{n} + 1) I
    \end{bmatrix} 
\end{align}
with $X \in \real^{n \times d}$ being the data matrix and $I$ being the identity matrix of appropriate dimension. They provide an $\eta_1$ for the objective function and it is obvious that constraints are convex (we ignore the dimension of the matrices for succinct representation). Thus,
\begin{align*}
    \eta_1((w, Z), (\widetilde{w}, \widetilde{Z})) = \begin{bmatrix} w - \widetilde{w} \\ M(\widetilde{w})^{-1} M(w) (Z - \widetilde{Z}) \end{bmatrix}, \; \eta_2((w, Z), (\widetilde{w}, \widetilde{Z})) = \begin{bmatrix} w - \widetilde{w} \\  Z - \widetilde{Z} \end{bmatrix} \;.
\end{align*}

We used these $\eta$ functions to construct updates and projection for Algorithm~\ref{alg:proj invex grad alg}. Let $f(w, Z) = \inner{M(w)}{Z}$. Let $	\nabla_w f(w, Z) = \frac{\partial \inner{M(w)}{Z}}{\partial w}$ and $\nabla_Z f(w, Z) = M(w)$, then using the $\eta$ functions and step-size $\alpha$ we write the following update steps for this problem:
\begin{align}
    \label{eq: fair lasso update}
    \begin{split}
        w_{t+1} = \prod_{\lambda}(w_{t} - \alpha \nabla_w f(w_t, Z_t)), \;
        \bar{Z}_{t+1} = Z_t - \alpha M(w_{t+1})^{-1}M(w_t) \nabla_Z f(w_t, Z_t)   \; ,
    \end{split}
\end{align}  
where $\prod_{\lambda}(\cdot)$ is the projection operator which uses soft thresholding to deal with $\ell_1$-regularization. We need to project $\bar{Z}_{t+1}$ on constraints to get the final $Z_{t+1}$.
\begin{align}
    \label{eq: fair lasso update Z}
    Z_{t+1} = \begin{matrix}
        \arg\min_{Z} & \| Z - 	\bar{Z}_{t+1} \|_F^2 &
        \text{such that} & \diag(Z) = 1, \; Z \succeq 0 
    \end{matrix}
\end{align}

\begin{figure}[!ht]
	\centering
	\begin{subfigure}{.45\textwidth}
		\centering
		\includegraphics[width=\linewidth]{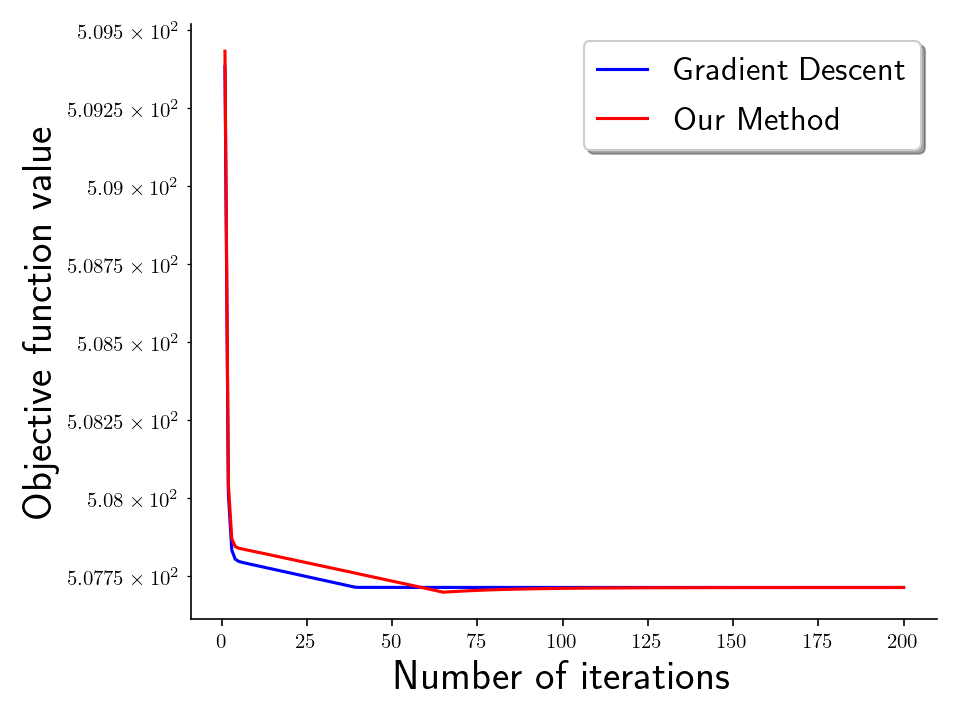}
		\caption{$n=500, d=100$}
	\end{subfigure}%
	\begin{subfigure}{.45\textwidth}
		\centering
		\includegraphics[width=\linewidth]{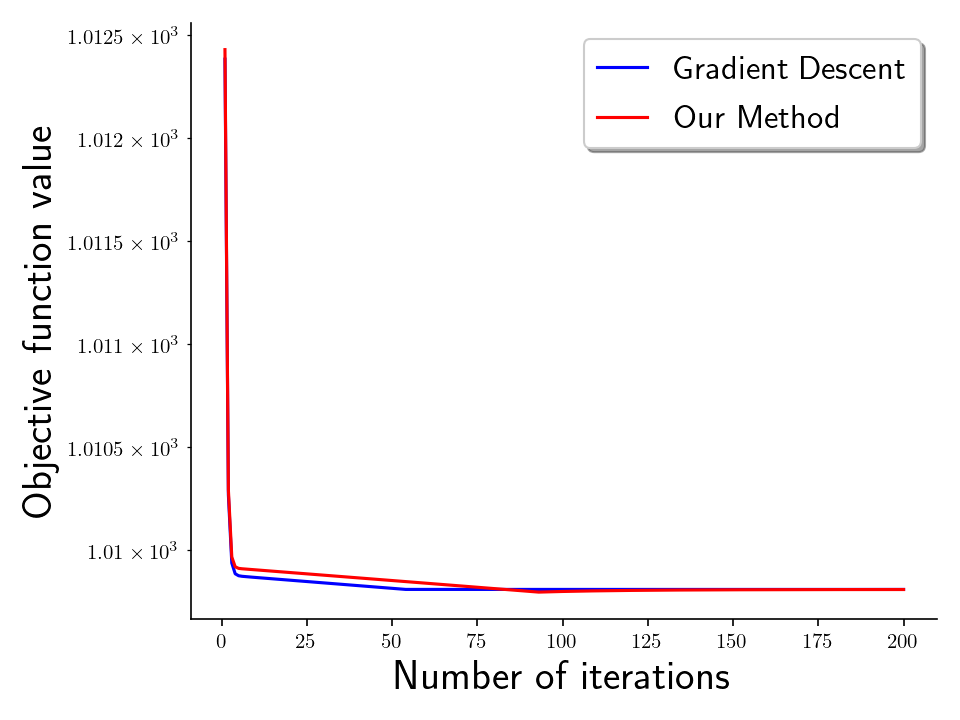}	
		\caption{$n=1000, d=200$}
	\end{subfigure}%
	\caption{\label{fig:fairlasso} Objective function value vs iterations for problem~\ref{eq: invex fair lasso} (semilog scale) - comparing Algorithm~\ref{alg:proj invex grad alg} with projected gradient descent algorithm for $n = 500, d = 100$ (left) and $ n = 1000, d = 200$ (right).}
\end{figure}  

We used update rules from equation~\eqref{eq: fair lasso update} and \eqref{eq: fair lasso update Z} to implement Algorithm~\ref{alg:proj invex grad alg}. We compared its performance against the projected gradient descent algorithm. The hyper-parameters (such as $\lambda$ and $\alpha$) and initial values of $w$ and $Z$ were kept the same across both algorithms. We report our results in Figure~\ref{fig:fairlasso}. We see that both algorithms perform in a similar manner. We expect this behavior as when $w_t$ is close to $w_{t+1}$ the update rules are the same for both the algorithms.

\subsection{Mixed linear regression}
\label{subsec:mixed linear regression}

In mixed linear regression, measurements come from one of the two different linear regression models and the task is to identify two regression vectors and the model associated with each data point. Mathematically, each data point is generated as follows: $y_i = z_i^* \inner{X_i}{\beta_1^*} + (1 - z_i^*) \inner{X_i}{\beta_2^*} + e_i, \forall i \in \{1, \cdots, n \}$ where $\beta_1^*$ and $\beta_2^*$ are $d$-dimensional vectors. \cite{barik2022sparse} proposed an invex program to solve this problem. Let $f(t, W, U) = \sum_{i=1}^n \frac{1}{2} \inner{S_i}{W + U} + \frac{1}{2} t_i \inner{S_i}{W - U}, g(t, W, U) = \| \vect(W) \|_1$ and $h(t, W, U) = \| \vect(U) \|_1$, where $	S_i = \begin{bmatrix}
		X_i \\ -y_i
\end{bmatrix} \begin{bmatrix}
		X_i^\T  -y_i
\end{bmatrix}$
and operator $\vect(.)$ vectorizes the matrix. Their invex formulation is given as:
\begin{align}
    \label{eq:invex mlr}
    \begin{matrix}
        \min_{t, W, U} & \sum_{i=1}^n f(t, W, U) + \lambda_1 g(t, W, U) + \lambda_2 h(t, W, U) \\
        \text{such that} &  W \succeq 0, \; U \succeq 0, \; W_{d+1, d+1} = 1, \; U_{d+1, d+1} = 1, \; \| t \|_{\infty} \leq 1
    \end{matrix}
\end{align}
	
The constraints of the problem are clearly convex. \cite{barik2022sparse} also provide an $\eta_1$ for the objective function, but it does not lend well to construct update rules required for Algorithm~\ref{alg:proj invex grad alg}.  We bypass this problem by showing that when $W \ne U$ then the objective function is invex for a different $\eta_1$. When $W = U$, we revert to the $\eta$ provided by \cite{barik2022sparse}. To that end, we prove the following lemma.
	
\begin{lemma}
    \label{lem:eta1 for mlr}
    Assume that $\widetilde{W} \ne \widetilde{U}$, then functions $f(t, W, U) = \sum_{i=1}^n \frac{1}{2} \inner{S_i}{W + U} + \frac{1}{2} t_i \inner{S_i}{W - U}, g(t, W, U) = \| \vect(W) \|_1$ and $h(t, W, U) = \| \vect(U) \|_1$ are $\eta_1$-invex for 
    \begin{align}
        \label{eq: eta1 for mlr}
        \eta_1((t, W, U), (\widetilde{t}, \widetilde{W}, \widetilde{U})) =  \begin{bmatrix} \tau \circ (t - \widetilde{t}) \\ W - \widetilde{W} \\ U - \widetilde{U}   \end{bmatrix} \; ,
    \end{align}
    where $\tau(W, U, \widetilde{W}, \widetilde{U}) \in \real^n$ such that $\tau_i = \frac{\inner{S_i}{W - U}}{\inner{S_i}{\widetilde{W} - \widetilde{U}}}$. 
\end{lemma}
Now we are ready to construct update and projection rules. 
Let $\nabla_t f(t, W, U)_i = \frac{1}{2} \inner{S_i}{W - U}, \forall i=\{1, \cdots, n\}$, $\nabla_W f(t, W, U) = \sum_{i=1}^n \frac{t_i + 1}{2} S_i$ and $\nabla_U f(t, W, U) = \sum_{i=1}^n \frac{1 - t_i}{2} S_i$, then we propose the following update steps for step-size $\alpha$:
\begin{align}
    \label{eq: mlr update}
    \begin{split}
        &\bar{W}_{k+1} = \prod_{\lambda_1}(W_{k} - \alpha \nabla_W f(t_k, W_k, U_k)), \;
        \bar{U}_{k+1} = \prod_{\lambda_2}(U_{k} - \alpha \nabla_W f(t_k, W_k, U_k))  \\
        &\bar{t}_{k+1} = t_k - \alpha \nabla_Z f(t_k, W_k, U_k) \oslash \tau(W_{k+1}, U_{k+1}, W_k, U_k)  \; ,
    \end{split}
\end{align}  
where $\prod_{\lambda}(\cdot)$ is the projection operator which uses soft thresholding to deal with $\ell_1$-regularization. We use the following projection steps to get $W_{k+1}, U_{k+1}$ and $t_{k+1}$.
\begin{align}
    \label{eq: mlr projection}
    \begin{split}
        &W_{k+1} = \begin{matrix}
            \arg\min_{W} & \| W - 	\bar{W}_{k+1} \|_F^2 \\
            \text{such that} & W_{d+1, d+1} = 1,\; W \succeq 0 
        \end{matrix}, \;     U_{k+1} = \begin{matrix}
            \arg\min_{U} & \| U - 	\bar{U}_{k+1} \|_F^2 \\
            \text{such that} & U_{d+1, d+1} = 1, \; U \succeq 0 
        \end{matrix} \\
        &t_{k+1} = \begin{matrix}
            \arg\min_{t} & \| t - 	\bar{t}_{k+1} \|_2^2 \\
            \text{such that} & \| t \|_{\infty} \leq 1
        \end{matrix} 
    \end{split}
\end{align}

\begin{figure}[!ht]
	\centering
	\begin{subfigure}{.45\textwidth}
		\centering
		\includegraphics[width=\linewidth]{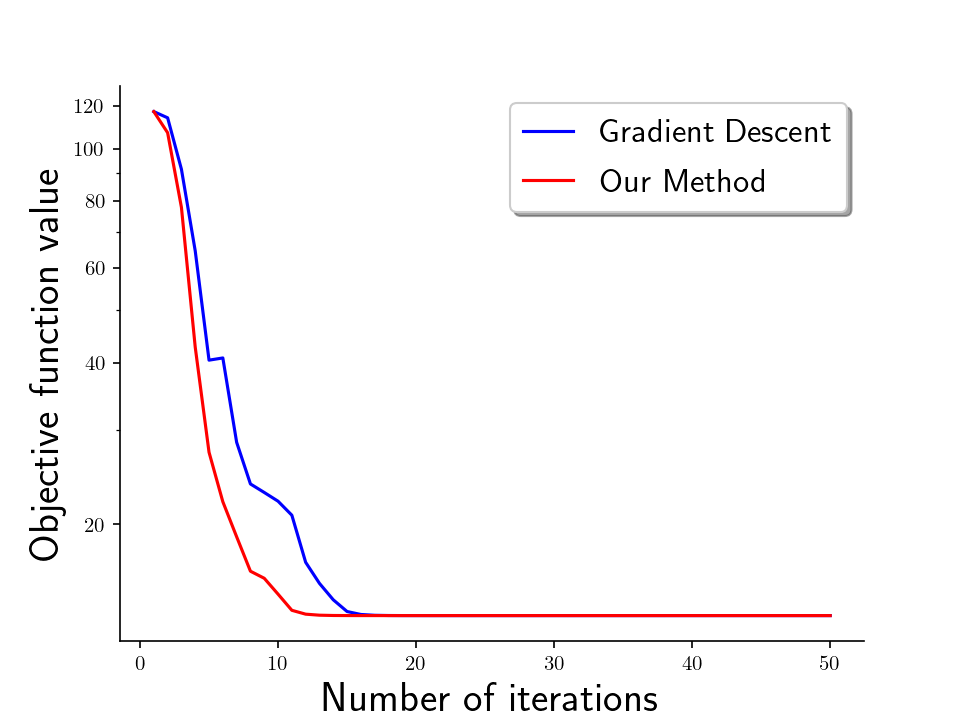}	
		\caption{$n=50, d=10$}
	\end{subfigure}%
 	\begin{subfigure}{.45\textwidth}
		\centering
		\includegraphics[width=\linewidth]{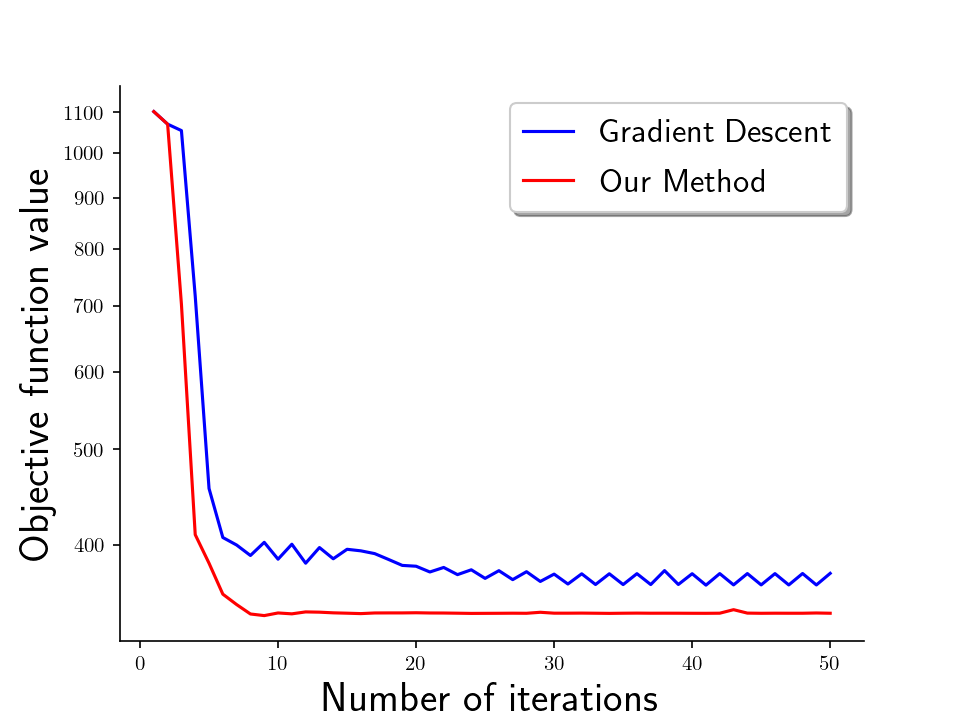}
		\caption{$n=500, d=100$}
	\end{subfigure}%
	\caption{\label{fig:mlr} Objective function value vs iterations for problem~\ref{eq:invex mlr} (semilog scale)- comparing Algorithm~\ref{alg:proj invex grad alg} with projected gradient descent algorithm for $n = 50, d = 10$ (left) and $ n = 500, d = 100$ (right).}
\end{figure}  

We implemented Algorithm~\ref{alg:proj invex grad alg} using update and projection rules from equation~\eqref{eq: mlr update} and equation~\eqref{eq: mlr projection}. Like before, we compared the performance of our algorithm with the projected gradient descent method with the same set of hyperparameters and initialization. We report our results in Figure~\ref{fig:mlr}. We see that our algorithm converges faster than the projected gradient descent algorithm.

\section{Conclusion and future work}
In this work, we have taken the first steps towards providing algorithms to solve constrained and unconstrained invex programs within certain technical conditions. We show that our algorithm can be used to solve constrained geodesically convex optimization problems with provable convergence rate guarantees. We also show the applicability of our proposed algorithm in a variety of machine-learning applications. Our work employs some natural assumptions for mathematical analysis. But these are only sufficient conditions for convergence. As the future direction, it would be interesting to see if these assumptions can be relaxed without losing on the convergence rate guarantees. From an application point of view, it would also be interesting to come up with an explicit form of update rules and the projection operator from a given $\eta$ for a large class of invex problems. Another direction of research could be to study the accelerated version of our algorithms. While already for the subclass of invex problems, namely, geodesically convex ones, it is known that without further assumptions/restrictions, global acceleration similar to the Euclidean Nesterov acceleration does not hold. However, it is a valuable question to explore conditions under which such an acceleration holds for our setting.

\bibliographystyle{apalike}
\setlength{\bibsep}{2pt}
\bibliography{invex_algo}

\begin{thebibliography}{}

\bibitem[Barik and Honorio, 2021]{barik2021fair}
Barik, A. and Honorio, J. (2021).
\newblock Fair sparse regression with clustering: An invex relaxation for a
  combinatorial problem.
\newblock {\em Advances in Neural Information Processing Systems},
  34:23245--23257.

\bibitem[Barik and Honorio, 2022]{barik2022sparse}
Barik, A. and Honorio, J. (2022).
\newblock Sparse mixed linear regression with guarantees: Taming an intractable
  problem with invex relaxation.
\newblock {\em International Conference on Machine Learning}.

\bibitem[Bello et~al., 2022]{bello2022dagma}
Bello, K., Aragam, B., and Ravikumar, P. (2022).
\newblock Dagma: Learning dags via m-matrices and a log-determinant acyclicity
  characterization.
\newblock {\em arXiv preprint arXiv:2209.08037}.

\bibitem[Ben-Israel and Mond, 1986]{ben1986invexity}
Ben-Israel, A. and Mond, B. (1986).
\newblock What is invexity?
\newblock {\em The ANZIAM Journal}, 28(1):1--9.

\bibitem[Boumal, 2020]{boumal2020introduction}
Boumal, N. (2020).
\newblock An introduction to optimization on smooth manifolds.
\newblock {\em Available online, May}, 3.

\bibitem[Boyd et~al., 2007]{boyd2007tutorial}
Boyd, S., Kim, S.-J., Vandenberghe, L., and Hassibi, A. (2007).
\newblock A tutorial on geometric programming.
\newblock {\em Optimization and engineering}, 8(1):67--127.

\bibitem[Ciliberto et~al., 2015]{ciliberto2015convex}
Ciliberto, C., Mroueh, Y., Poggio, T., and Rosasco, L. (2015).
\newblock Convex learning of multiple tasks and their structure.
\newblock In {\em International Conference on Machine Learning}, pages
  1548--1557. PMLR.

\bibitem[Dinuzzo et~al., 2011]{dinuzzo2011learning}
Dinuzzo, F., Ong, C.~S., Pillonetto, G., and Gehler, P.~V. (2011).
\newblock Learning output kernels with block coordinate descent.
\newblock In {\em Proceedings of the 28th International Conference on Machine
  Learning (ICML-11)}, pages 49--56.

\bibitem[Duchi et~al., 2010]{duchi2010composite}
Duchi, J.~C., Shalev-Shwartz, S., Singer, Y., and Tewari, A. (2010).
\newblock Composite objective mirror descent.
\newblock In {\em COLT}, volume~10, pages 14--26. Citeseer.

\bibitem[Hanson, 1981]{hanson1981sufficiency}
Hanson, M.~A. (1981).
\newblock On sufficiency of the kuhn-tucker conditions.
\newblock {\em Journal of Mathematical Analysis and Applications},
  80(2):545--550.

\bibitem[Hinder et~al., 2020]{hinder2020near}
Hinder, O., Sidford, A., and Sohoni, N. (2020).
\newblock Near-optimal methods for minimizing star-convex functions and beyond.
\newblock In {\em Conference on learning theory}, pages 1894--1938. PMLR.

\bibitem[Jain et~al., 2017]{jain2017non}
Jain, P., Kar, P., et~al. (2017).
\newblock Non-convex optimization for machine learning.
\newblock {\em Foundations and Trends{\textregistered} in Machine Learning},
  10(3-4):142--363.

\bibitem[Karimi et~al., 2016]{karimi2016linear}
Karimi, H., Nutini, J., and Schmidt, M. (2016).
\newblock Linear convergence of gradient and proximal-gradient methods under
  the polyak-{\l}ojasiewicz condition.
\newblock In {\em Joint European conference on machine learning and knowledge
  discovery in databases}, pages 795--811. Springer.

\bibitem[Liu and Boumal, 2020]{liu2020simple}
Liu, C. and Boumal, N. (2020).
\newblock Simple algorithms for optimization on riemannian manifolds with
  constraints.
\newblock {\em Applied Mathematics \& Optimization}, 82:949--981.

\bibitem[Lozano et~al., 2016]{lozano2016minimum}
Lozano, A.~C., Meinshausen, N., and Yang, E. (2016).
\newblock Minimum distance lasso for robust high-dimensional regression.

\bibitem[Luo and Trillos, 2022]{luo2022nonconvex}
Luo, Y. and Trillos, N.~G. (2022).
\newblock Nonconvex matrix factorization is geodesically convex: Global
  landscape analysis for fixed-rank matrix optimization from a riemannian
  perspective.
\newblock {\em arXiv preprint arXiv:2209.15130}.

\bibitem[Mohan and Neogy, 1995]{mohan1995invex}
Mohan, S. and Neogy, S. (1995).
\newblock On invex sets and preinvex functions.
\newblock {\em Journal of Mathematical Analysis and Applications},
  189(3):901--908.

\bibitem[Noor, 2005]{noor2005invex}
Noor, M.~A. (2005).
\newblock Invex equilibrium problems.
\newblock {\em Journal of Mathematical Analysis and Applications},
  302(2):463--475.

\bibitem[Pini, 1994]{pini1994convexity}
Pini, R. (1994).
\newblock Convexity along curves and indunvexity.
\newblock {\em Optimization}, 29(4):301--309.

\bibitem[Pinilla et~al., 2022]{pinilla2022improved}
Pinilla, S., Mu, T., Bourne, N., and Thiyagalingam, J. (2022).
\newblock Improved imaging by invex regularizers with global optima guarantees.
\newblock {\em arXiv preprint arXiv:2211.10112}.

\bibitem[Udriste, 2013]{udriste2013convex}
Udriste, C. (2013).
\newblock {\em Convex functions and optimization methods on Riemannian
  manifolds}, volume 297.
\newblock Springer Science \& Business Media.

\bibitem[Wang et~al., 2022]{wang2022provable}
Wang, J.-K., Lin, C.-H., Wibisono, A., and Hu, B. (2022).
\newblock Provable acceleration of heavy ball beyond quadratics for a class of
  polyak-lojasiewicz functions when the non-convexity is averaged-out.
\newblock In {\em International Conference on Machine Learning}, pages
  22839--22864. PMLR.

\bibitem[Weber and Sra, 2022]{weber2022riemannian}
Weber, M. and Sra, S. (2022).
\newblock Riemannian optimization via frank-wolfe methods.
\newblock {\em Mathematical Programming}, pages 1--32.

\bibitem[Zhang and Sra, 2016]{zhang2016first}
Zhang, H. and Sra, S. (2016).
\newblock First-order methods for geodesically convex optimization.
\newblock In {\em Conference on Learning Theory}, pages 1617--1638. PMLR.

\bibitem[Zhang et~al., 2020]{zhang2020variational}
Zhang, J., Koppel, A., Bedi, A.~S., Szepesvari, C., and Wang, M. (2020).
\newblock Variational policy gradient method for reinforcement learning with
  general utilities.
\newblock {\em Advances in Neural Information Processing Systems},
  33:4572--4583.

\bibitem[Zhang et~al., 2022]{zhang2022minimax}
Zhang, P., Zhang, J., and Sra, S. (2022).
\newblock Minimax in geodesic metric spaces: Sion's theorem and algorithms.
\newblock {\em arXiv preprint arXiv:2202.06950}.

\end{thebibliography}

\newpage

\appendix

\section{Functions satisfying PL inequality}
\label{sec:functions satisfying pl inequality}

PL functions are a special class of (possibly nonconvex) functions that satisfy the following property:
 \begin{align*}
        \| \nabla f(x) \|^2 \geq \mu (f(x) - f(x^*)) \; ,
\end{align*}
\cite{karimi2016linear} showed that if an $L$-smooth function satisfies PL inequality, then it can be shown that it achieves an exponential convergence rate. These functions are known to be invex. Here we provide a characterization of their invexity by providing an $\eta$ which can be used to construct updates in Algorithm~\ref{alg:invex grad alg}. To that end, we show the validity of Lemma~\ref{lem:pl inequality}.
\paragraph{Lemma~\ref{lem:pl inequality}}
    \emph{Let $f(x)$ be an $L$-smooth function which satisfies PL inequality for some $\mu > 0$. Then it is $\eta$-invex for the following $\eta$:
    \begin{align}
        \eta(y, x) = \frac{1}{\mu}\left(\nabla f(y) + L \frac{\| y - x \|}{\| \nabla f(x) \|}   \nabla f(x) \right) 
    \end{align}
}
\begin{proof}
    \label{proof:lem pl inequality}
    Since $f(x)$ follows PL inequality for some $\mu > 0$, the following inequality holds~\citep{karimi2016linear} for all $x$ in the domain $D$ of $f(x)$:
    \begin{align}
        \label{eq:pl inequality}
        \| \nabla f(x) \|^2 \geq \mu (f(x) - f(x^*)) \; ,
    \end{align}
    where $x^* = \arg\min_{x \in D} f(x)$. Using Taylor series expansion of $g(y) = \nabla f(y)^\T v$ around $x$ and then substituting for $v = \nabla f(x)$, we can write
    \begin{align}
        \| \nabla f(x) \|^2  = \nabla f(x)^\T \nabla f(y) + \nabla f(x)^\T \nabla^2 f(z) (x - y)  
    \end{align}
    for some $z= y + t (x - y), t \in [0, 1]$. It follows that 
    \begin{align}
        \inner{\nabla f(x)}{   \nabla f(y) +  \nabla^2 f(z) (x - y) } \geq \mu (f(x) - f(x^*)) \geq \mu (f(x) - f(y)) 
    \end{align}
    Given that $f(x)$ is $L$-smooth and 
    \begin{align}
        \max_{\| M \|_2 \leq L} u^\T M v = L \| u \| \| v \| \; ,
    \end{align}
    we can write
    \begin{align}
        \inner{\nabla f(x)}{   \nabla f(y) +  L \frac{\| y - x \|}{\| \nabla f(x) \|} \nabla f(x) } \geq \mu (f(x) - f(x^*)) \geq \mu (f(x) - f(y)) 
    \end{align}
    This completes our proof.
\end{proof}

\section{Quasar convex functions}
\label{sec:quasar convex functions}

Quasar convex functions~\citep{hinder2020near} are another interesting class of possibly nonconvex functions which achieve global minima at all their stationary points. Thus, they fall under the class of invex functions. Here, we propose an $\eta$ for the class of quasar convex functions which can be used to construct updates in Algorithm~\ref{alg:invex grad alg}. Below, we prove Lemma~\ref{lem:quasar convex functions}.
\paragraph{Lemma~\ref{lem:quasar convex functions}}
    \emph{For any $\nu \geq 0$, there exists a $\beta \in [0, 1]$ such that quasar convex functions are $\eta$-invex for $\eta(y, x) = \frac{\beta}{\nu(1 - \beta)} (y - x) $. 
}
\begin{proof}
    \label{proof:lem quasar convex functions}
    First, we use the result from Lemma 2 of \cite{hinder2020near} to show that for any $\nu \geq 0$, there exists a $\beta \in [0, 1]$, such that 
    \begin{align}
        \label{eq:lem 2 quasar}
        \beta \nabla f(x)^\T (y - \frac{x - \beta y}{1 - \beta}) \leq \nu (f(y) - f(x)) 
    \end{align}
    We can simplify equation~\eqref{eq:lem 2 quasar} to write
    \begin{align}
        f(y) \geq f(x) + \nabla f(x)^\T \frac{\beta}{\nu (1 - \beta)} (y - x)
    \end{align}
    This completes our proof.
\end{proof}
We note that $\beta$ can be computed efficiently using a binary-search algorithm(refer to Algorithm 2 of \cite{hinder2020near}).

\section{Proof of Theorems and Lemmas}
\label{sec:proof of theorems and lemmas}

Before we begin to solve the optimization problem~\eqref{eq:invex program}, we will prove some technical results to understand the problem in a better way. First, we will show that the constraint set is indeed an $\eta_2$-invex set. We will do it in two parts.
	
\begin{lemma}
    \label{lem:invex set from a function}
    Let $\phi:\calC \to \real$ be an $\eta$-invex function, then $\phi(x) \leq c$ is an $\eta$-invex set for any $c \in \real$.
\end{lemma}

\begin{proof}
 Let $\gamma_{x, y}$ be the underlying curve connecting $x, y \in \calC$ corresponding to $\eta(y, x)$ satisfying equation~\eqref{eq:itegrable gamma}. 
 
 Using definition~\ref{def:invex set}, we can redefine invex functions on an invex set $\calA \subseteq \calC$ as following:

\begin{definition}[Invex functions on invex set] 
    \label{def:invex functions on invex set}
    Let $\calA \subseteq \calC$ be an invex set. A real-valued differentiable function $\phi:\calA \to \real$ is called invex if 
    \begin{align}
        \label{eq:invex functions on invex set}
	\phi(\gamma_{x, y}(t)) \leq (1 - t) \phi(x) + t \phi(y), \; \forall x, y \in \calA, \forall t \in [0, 1] 
    \end{align}
\end{definition}
	
Definitions \ref{def:invex functions} and \ref{def:invex functions on invex set} are connected with each other through the relationship between $\gamma_{x, y}$ and $\eta(y, x)$ in equation~\eqref{eq:itegrable gamma}. 

 Let $\calA = \{ x \in \calC | \phi(x) \leq c \}$. We take $x, y \in \calA$. We will then need to show that $\gamma_{x, y}(t) \in \calA, \forall t \in [0, 1]$. Using definition                        
\ref{def:invex functions on invex set}
\begin{align*}
	\phi(\gamma_{x,y}(t)) &\leq (1 - t) \phi(x) + t \phi(y), \; \forall x, y \in \calA, \forall t \in [0, 1] \\
	&\leq (1 - t) c + t  c = c
\end{align*} 
It follows that $\gamma_{x, y}(t) \in \calA, \forall t \in [0, 1]$.
\end{proof}

Next, we use Lemma~\ref{lem:invex set from a function} to show that the constraint set is an $\eta$-invex set.
	
\begin{lemma}
    \label{lem:intersection of invex sets}
    Let $g_i:\calC \to \real, \forall i \in \{ 1, \cdots, m \}$ be $\eta$-invex functions, then the set $\calA = \cap_{i=1}^m \calA_i$ is $\eta$-invex where $\calA_i = \{ x \in \calC | g_i(x) \leq 0 \}$. 
\end{lemma}
\begin{proof}
	Let $x, y \in \calA$, then by definition $x, y \in \calA_i, \forall i \in \{ 1, \cdots, m \}$. We know from Lemma~\ref{lem:invex set from a function}, that $\calA_i$'s are $\eta$-invex set. Let $\gamma_{x, y}$ be the underlying curve connecting $x, y$. Then, it follows that $\gamma_{x, y}(t) \in \calA_i, \forall i \in \{1, \cdots, m\}, \forall t \in [0, 1]$. Thus, $\gamma_{x, y}(t) \in \calA, \forall t \in [0, 1]$. 
\end{proof}

\paragraph{Theorem~\ref{thm:convergence of L smooth functions}}(Convergence of $L$-smooth functions.)
\emph{	Let $f$ be a $L$-smooth function and $f^* = \min_{x \in \calC} f(x) \geq B$ for some $B > - \infty$. If $x_k$ is computed using Algorithm~\ref{alg:invex grad alg} with $\alpha_k = \alpha \in (0, \frac{2}{L})$, then 
	\begin{align}
		\lim_{k \to \infty} \| \nabla f(x_k) \| = 0 \; ,
	\end{align} 
}
\begin{proof}
	\label{proof:convergence of L smooth functions}
	Since $f$ is an $L$-smooth function. We have
	\begin{align}
		\label{eq:f is l smooth}
		\begin{split}
			f(x_{k+1}) \leq &f(x_k) + \inner{\eta(x_{k+1}, x_k)}{\nabla f(x_k)}_{x_k} + \frac{L}{2} \| \eta(x_{k+1}, x_k) \|^2
		\end{split}
	\end{align}
	Using Algorithm~\ref{alg:invex grad alg}, we have $\eta(x_{k+1}, x_k) = - \alpha_k \nabla f(x_k)$. Thus,
	\begin{align}
		\label{eq:f is l smooth 1}
		\begin{split}
			f(x_{k+1}) \leq &f(x_k) - \alpha \inner{ \nabla f(x_k) }{\nabla f(x_k)}_{x_k} + \frac{L \alpha^2}{2} \| \nabla f(x_k) \|^2
		\end{split}
	\end{align}
	Since $\alpha \in (0, \frac{2}{L})$, it follows that
	\begin{align}
		\label{eq:decreasing function}
		\alpha (1 - \frac{L\alpha}{2}) \| \nabla f(x_k) \|^2 \leq f(x_k) - f(x_{k+1})
	\end{align}
	After telescoping sum and simplification, we get
	\begin{align}
		\label{eq:f is l smooth 2}
		\sum_{k=0}^{\infty} \| \nabla f(x_k) \|^2 \leq \frac{f(x_0) - B}{\alpha (1 - \frac{L \alpha}{2})}
	\end{align}
	Since right-hand side of equation~\eqref{eq:f is l smooth 2} is finite, it follows that $\lim_{k \to \infty} \| \nabla f(x_k) \| = 0$.
\end{proof} 

\paragraph{Theorem \ref{thm:convergence of invex f with triangle inequality}}(Convergence of invex functions.)
	\emph{Let $f: \calC \to \real$ be an $L$-smooth $\eta$-invex function such that $\calC$ satisfies Assumption~\ref{assum:triangle inequality}. Furthermore, let $x^* = \arg\min_{x \in \calC} f(x)$ such that $f(x^*) > -\infty$ and $\| \eta(x_0, x^*) \| \leq M < \infty$. If $x_k$ is computed using Algorithm~\ref{alg:invex grad alg} with $\alpha_k = \alpha \in (0,  \frac{2}{L})$, then $f(x_k)$ converges to $f(x^*)$ at the rate $\calO(\frac{1}{k})$.  
}
\begin{proof}
	\label{proof:convergence of invex f with triangle inequality}
	We apply Equation~\eqref{eq:triangle inequality} by taking $x = x_k, y = x_{k+1}$ and $z = x^*$.
	\begin{align}
		\label{eq:triangle inequality 1}
		\begin{split}
			\| \eta(x_{k+1}, x^*) \|^2 \leq & \| \eta(x_k, x^*) \|^2 + b \| \eta(x_{k+1}, x_k) \|^2 -  c \inner{ \eta(x_{k+1}, x_k) }{ \eta(x^*, x_k)}_{x_k} 
		\end{split}
	\end{align}
	From Algorithm~\ref{alg:invex grad alg}, $\eta(x_{k+1}, x_k) = -\alpha \nabla f(x_k)$.
	\begin{align}
		\label{eq:triangle inequality 2}
		\begin{split}
			\| \eta(x_{k+1}, x^*) \|^2 \leq & \| \eta(x_k, x^*) \|^2 + b \alpha^2 \|  \nabla f(x_k) \|^2 +  c \alpha \inner{ \nabla f(x_k) }{ \eta(x^*, x_k)}_{x_k} 
		\end{split}
	\end{align}
	Note that since $f$ is $\eta$-invex, we have $f(x^*) \geq f(x_k) + \inner{\eta(x^*, x_k)}{\nabla f(x_k)}$. Thus,
	\begin{align}
		\label{eq:triangle inequality 3}
		\begin{split}
			&\| \eta(x_{k+1}, x^*) \|^2 \leq  \| \eta(x_k, x^*) \|^2 + b \alpha^2 \|  \nabla f(x_k) \|^2 +  c \alpha (f(x^*) - f(x_k)) \\
			&c \alpha (f(x_k) - f(x^*)) \leq \| \eta(x_k, x^*) \|^2 - \| \eta(x_{k+1}, x^*) \|^2 + b \alpha^2 \|  \nabla f(x_k) \|^2
		\end{split}
	\end{align}
	Summing over $T$ terms, we get
	\begin{align}
		\label{eq:triangle inequality 4}
		\begin{split}
			&c \alpha \sum_{k=0}^T (f(x_k) - f(x^*)) \leq \| \eta(x_0, x^*) \|^2 - \| \eta(x_{T+1}, x^*) \|^2 + b \alpha^2 \sum_{k=0}^T \|  \nabla f(x_k) \|^2
		\end{split}
	\end{align}
	Since $\alpha \in  (0, \frac{2}{L}) $, we make two observations from equation~\eqref{eq:decreasing function},
	\begin{align}
		\label{eq:observation from thm 1}
		\begin{split}
			\sum_{k=0}^T \| \nabla f(x_k) \|^2 &\leq \frac{ f(x_0) - f(x_{T+1})}{\alpha (1 - \frac{L\alpha}{2})} \\
			f(x_{k+1}) - f(x_k) &\leq \alpha (-1 + \frac{L\alpha}{2}) \| \nabla f(x_k) \|^2 \leq 0
		\end{split}
	\end{align}
	By using the observations in equation~\eqref{eq:observation from thm 1}, it follows that 
	\begin{align}
		\label{eq:observation from thm 2}
		\begin{split}
			c \alpha T ( f(x_T) - f(x^*) ) &\leq \| \eta(x_0, x^*)  \|^2 +  \frac{b \alpha (f(x_0) - f(x^*)) }{1 - \frac{L\alpha}{2}} \;.
		\end{split}
	\end{align}
	Note that using $L$-smoothness condition and noticing that $\nabla f(x^*) = 0$, we can show that 
	\begin{align}
		\label{eq:using L smooth}
		f(x_0) - f(x^*) \leq \frac{L}{2} \| \eta(x_0, x^*) \|^2
	\end{align}
	Thus,
	\begin{align}
		\begin{split}
			f(x_T) - f(x^*)  &\leq \frac{1}{T} \frac{1}{c\alpha} (1+
			\frac{b \alpha L }{2 - L\alpha}) M^2
		\end{split}
	\end{align}
	This proves our claim.
\end{proof}

\paragraph{Theorem \ref{thm:convergence of strongly invex functions}}(Convergence of strongly invex functions.)
\emph{Let $f:\calC \to \real$ be an $L$-smooth $\mu$-strongly $\eta$-invex function such that $\calC$ satisfies Assumption~\ref{assum:triangle inequality}. Furthermore, let $x^* = \arg\min_{x \in \calC} f(x)$ such that $f(x^*) > -\infty,  \| \eta(x_0, x^*) \| \leq M < \infty$ and $\| \eta(y, x) \|_x^2 \leq R \| \eta(x, y) \|_y^2, \forall x, y \in \calC$ for some $R > 0$. If $x_k$ is computed using Algorithm~\ref{alg:invex grad alg} with $\alpha_k = \alpha \in (0,  \min (\frac{2}{ R \mu c }, \frac{c}{2bL}))$, then 
\begin{align}
	\label{eq:convergence of strongly invex}
	\| \eta(x_{k+1} , x^*) \|^2 \leq (1 - \frac{c \alpha R \mu}{2})^{k+1} M^2 \; .
\end{align}
}
\begin{proof}
	\label{proof:convergence of strongly invex functions}
	We begin our proof by proving an auxiliary lemma.
	\begin{lemma}
		\label{lem:L-smooth auxiliary}
		If $f$ is $L$-smooth then 
		\begin{align}
			f(x^*) - f(x) \leq -\frac{1}{2L} \| \nabla f(x) \|^2, \forall x \in \calC \; .
		\end{align}
	\end{lemma}
	\begin{proof}
		\label{proof:L-smooth auxiliary}
		We can always find a $y \in \calC$ such that $\eta(y, x) = - \frac{1}{L} \nabla f(x)$. Using equation~\eqref{eq:L-smooth} we have,
		\begin{align}
			\label{eq:l smooth auxiliary}
			\begin{split}
				f(y) &\leq f(x) - \frac{1}{L} \inner{\nabla f(x)}{ \nabla f(x)}_x + \frac{L}{2} \| \frac{1}{L} \nabla f(x) \|^2\\
				&=f(x) - \frac{1}{2L} \| \nabla f(x) \|^2
			\end{split}
		\end{align}
		Clearly, $f(x^*) \leq f(y)$, thus
		\begin{align}
			f(x^*) - f(x) \leq -\frac{1}{2L} \| \nabla f(x) \|^2 \; .
		\end{align}
		This proves our claim.
	\end{proof}
	
	Now using Assumption~\ref{assum:triangle inequality} for $x = x_k, y = x_{k+1}$ and $z = x^*$, we have
	\begin{align}
		\label{eq:convergence of strongly invex 1}
		\begin{split}
			\| \eta(x_{k+1}, x^*) \|^2 \leq & \| \eta(x_k, x^*) \|^2 + b \| \eta(x_{k+1}, x_k) \|^2 -  c \inner{ \eta(x_{k+1}, x_k) }{ \eta(x^*, x_k)}_{x_k} 
		\end{split}
	\end{align}
	Now since $\eta(x_{k+1}, x_k) = -\alpha \nabla f(x_k)$, we have
	\begin{align}
		\label{eq:convergence of strongly invex 2}
		\begin{split}
			\| \eta(x_{k+1}, x^*) \|^2 \leq & \| \eta(x_k, x^*) \|^2 + b \alpha^2 \| \nabla f(x_k) \|^2 +   c \alpha \inner{ \nabla f(x_k) }{ \eta(x^*, x_k)}_{x_k} 
		\end{split}
	\end{align}
	Using the strong convexity of $f$ and setting $y = x^*$ and $x = x_k$, we have  
	\begin{align}
		\label{eq:convergence of strongly invex 3}
		\begin{split}
			\| \eta(x_{k+1}, x^*) \|^2 \leq & \| \eta(x_k, x^*) \|^2 + b \alpha^2 \| \nabla f(x_k) \|^2 +   c \alpha (f(x^*) - f(x_k) - \frac{\mu}{2} \| \eta(x^*, x_k) \|^2) 
		\end{split}
	\end{align}
	Using the condition that $\| \eta(x^*, x_k) \|^2 \leq R \| \eta(x_k, x^*) \|^2$, we have
	\begin{align}
		\label{eq:convergence of strongly invex 4}
		\begin{split}
			\| \eta(x_{k+1}, x^*) \|^2 \leq & (1 -\frac{c \alpha R \mu}{2} ) \| \eta(x_k, x^*) \|^2 + b \alpha^2 \| \nabla f(x_k) \|^2 +  c \alpha (f(x^*) - f(x_k)) 
		\end{split}
	\end{align}
	Using $L$-smoothness of $f$ and Lemma~\ref{lem:L-smooth auxiliary}, we have
	\begin{align}
		\label{eq:convergence of strongly invex 5}
		\begin{split}
			\| \eta(x_{k+1}, x^*) \|^2 \leq & (1 -\frac{c \alpha R \mu}{2} ) \| \eta(x_k, x^*) \|^2 - \alpha   ( - 2 b \alpha  L  +  c )(f(x_k) - f(x^*)) 
		\end{split}
	\end{align}
	Taking $\alpha \leq \min (\frac{2}{ R \mu c }, \frac{c}{2bL}) $, we get
	\begin{align}
		\label{eq:convergence of strongly invex 6}
		\begin{split}
			\| \eta(x_{k+1}, x^*) \|^2 \leq & (1 -\frac{c \alpha R \mu}{2} ) \| \eta(x_k, x^*) \|^2 
		\end{split}
	\end{align}
	
	We prove our result by unrolling the recurrence in equation~\eqref{eq:convergence of strongly invex 6}.
\end{proof}

\subsection{Convergence of Projected Invex Gradient Descent Method}
\label{subsec:convergence of projected invex gradient method}
						
Next, we will show that once Assumption~\ref{assum:contraction} is satisfied by the projection operator, results from Theorems~\ref{thm:convergence of invex f with triangle inequality} and \ref{thm:convergence of strongly invex functions} extend nicely to Algorithm~\ref{alg:proj invex grad alg}. 

\begin{theorem}
    \label{thm:convergence of invex f with triangle inequality projected}
    Let $f: \calC \to \real$ be an $L$-smooth $\eta_1$-invex function such that $\calC$ satisfies Assumption~\ref{assum:triangle inequality}. Let $\calA \subseteq \calC$ be an $\eta_2$-invex set. Furthermore, let $x^* = \arg\min_{x \in \calA} f(x)$ such that $f(x^*) > -\infty$ and $\| \eta_1(x_0, x^*) \| \leq M < \infty$. If $x_k$ is computed using Algorithm~\ref{alg:proj invex grad alg} with $\alpha_k = \alpha \in (0,  \frac{2}{L})$ and with projection operator satisfying Assumption~\ref{assum:contraction}, then $f(x_k)$ converges to $f(x^*)$ at the rate $\calO(\frac{1}{k})$.  
\end{theorem}
\begin{proof}
	\label{proof:convergence of invex f with triangle inequality projected}
	First notice that since $x^* \in \calA$, we have $\rho_{\eta_2}(x^*) = x^*$. We follow the same proof technique as Theorem~\ref{thm:convergence of invex f with triangle inequality} until equation~\eqref{eq:triangle inequality 3} which becomes:
	\begin{align}
		\label{eq:triangle inequality 1 proj}
		\begin{split}
			&c \alpha (f(x_k) - f(x^*)) \leq \| \eta_1(x_k, x^*) \|^2 - \| \eta_1(y_{k+1}, x^*) \|^2 + b \alpha^2 \|  \nabla f(x_k) \|^2
		\end{split}
	\end{align}
	Using Assumption~\ref{assum:contraction}, we know that $\| \eta_1(y_{k+1}, x^*) \|^2 \geq  \| \eta_1(\rho_{\eta_2}(y_{k+1}), x^*) \|^2  = \| \eta_1(x_{k+1}, x^*) \|^2 $ and thus remaining steps of the proof follow.
\end{proof}

 We have a similar result for $\mu$-strongly $\eta$-invex functions.

\begin{theorem}
    \label{thm:convergence of strongly invex functions projected}
    Let $f:\calC \to \real$ be an $L$-smooth $\mu$-strongly $\eta_1$-invex function such that $\calC$ satisfies Assumption~\ref{assum:triangle inequality}. Let $\calA \subseteq \calC$ be an $\eta_2$-invex set. Furthermore, let $x^* = \arg\min_{x \in \calA} f(x)$ such that $f(x^*) > -\infty,  \| \eta_1(x_0, x^*) \| \leq M < \infty$ and $\| \eta_1(y, x) \|_x^2 \leq R \| \eta_1(x, y) \|_y^2, \forall x, y \in \calC$ for some $R > 0$. If $x_k$ is computed using Algorithm~\ref{alg:proj invex grad alg} with $\alpha_k = \alpha \in (0,  \min (\frac{2}{ R \mu c }, \frac{c}{2bL}))$ and with projection operator satisfying Assumption~\ref{assum:contraction}, then 
    \begin{align}
        \label{eq:convergence of strongly invex proj}
        \| \eta_1(x_{k+1} , x^*) \|^2 \leq \left(1 - \frac{c \alpha R \mu}{2}\right)^{k+1} M^2 \; .
    \end{align}
\end{theorem}
\begin{proof}
	\label{proof:convergence of strongly invex functions projected}
	Again, we notice that since $x^* \in \calA$, we have $\rho_{\eta_2}(x^*) = x^*$. We follow the same proof technique as Theorem~\ref{thm:convergence of strongly invex functions} until equation~\eqref{eq:convergence of strongly invex 6} which becomes:
	\begin{align}
		\label{eq:convergence of strongly invex 6 proj}
		\begin{split}
			\| \eta_1(y_{k+1}, x^*) \|^2 \leq & (1 -\frac{c \alpha R \mu}{2} ) \| \eta_1(x_k, x^*) \|^2 
		\end{split}
	\end{align}
	Using Assumption~\ref{assum:contraction}, we note that $\| \eta_1(x_{k+1}, x^*) \|^2  = \| \eta_1(\rho_{\eta_2}(y_{k+1}), x^*) \|^2 \leq \| \eta_1(y_{k+1}, x^*) \|^2 $ and thus remaining steps of the proof follow.	
\end{proof}

\paragraph{Theorem \ref{thm: contraction geodesically convex}}(Projection contracts for geodesically convex sets with negative curvature.)
\emph{
	Let $\rho$ be the projection operator as defined in~\ref{def: Projection on invex set} on a closed geodesically convex subset $\calA$  of a simply connected Riemannian manifold $\calC$ with sectional curvature upper bounded by $\kappa \leq 0$. Then the projection satisfies Assumption~\ref{assum:contraction}.
}
\begin{proof}
	\label{proof: contraction geodesically convex} 
	First note that since geodesics are constant velocity curves, there exists a parameterization such that $d(y, x) = \| \dot{\gamma}_{x, y}(0) \| = \| \eta(y, x) \|$ where $d(y, x)$ is the length of geodesic between $x$ and $y$. Thus, it only remains to show that $d(y, x)$ contracts for geodesically convex sets (sometimes known as totally convex sets) which follows from Lemma 11.2 of \cite{udriste2013convex}.
\end{proof}			
				
\paragraph{Lemma \ref{lem: eta for h(W)}}
\emph{The function $h(W) = - \log \det (I - W \circ W), \forall W \in \mathcal{W}$ is $\eta$-invex for $\eta(U, W) = -\frac{1}{2} ((I - W \circ W) (\log(I - U \circ U) - \log(I - W \circ W))) \oslash W$.}
\begin{proof}
	\label{proof: eta for h(W)}
	The invexity of $h(W)$ is already shown by~\cite{bello2022dagma}. Here, we will verify that our proposed $\eta$ satisfies equation~\eqref{eq:invexity}. Note that $\nabla h(W) = 2 (I - W \circ W)^{-\T} \circ W$. Then $\forall U, W \in \mathcal{W}$,
	\begin{align*}
		&h(U) - h(W) - \inner{\eta(U, W)}{\nabla h(W)} = - \log \det (I - U \circ U) +  \log \det (I - W \circ W) - \\
		& \inner{ -\frac{1}{2} ((I - W \circ W) (\log(I - U \circ U) - \log(I - W \circ W))) \oslash W}{2 (I - W \circ W)^{-\T} \circ W} \\
		&= 0
	\end{align*}
	This validates our claim.
\end{proof}
			
\paragraph{Lemma \ref{lem:eta1 for mlr}}
\emph{	Assume that $\widetilde{W} \ne \widetilde{U}$, then functions $f(t, W, U) = \sum_{i=1}^n \frac{1}{2} \inner{S_i}{W + U} + \frac{1}{2} t_i \inner{S_i}{W - U}, g(t, W, U) = \| \vect(W) \|_1$ and $h(t, W, U) = \| \vect(U) \|_1$ are $\eta_1$-invex for 
	\begin{align}
		\eta_1((t, W, U), (\widetilde{t}, \widetilde{W}, \widetilde{U})) =  \begin{bmatrix} \tau \circ (t - \widetilde{t}) \\ W - \widetilde{W} \\ U - \widetilde{U}   \end{bmatrix} \; ,
	\end{align}
	where $\tau(W, U, \widetilde{W}, \widetilde{U}) \in \real^n$ such that $\tau_i = \frac{\inner{S_i}{W - U}}{\inner{S_i}{\widetilde{W} - \widetilde{U}}}$. 
}
\begin{proof}

	The invexity of the objective function in \eqref{eq:invex mlr} is already shown by~\cite{barik2022sparse}. It suffices to verify that our proposed $\eta$ satisfies equation~\eqref{eq:invexity} for $f(t, W, U), g(t, W, U)$ and $h(t, W, U)$. It can be trivially verified that our $\eta_1$ works for $g(t, W, U)$ and $h(t, W, U)$ due to their convexity. For $f(t, W, U)$,
	\begin{align*}
		\frac{\partial f}{\partial t_i} &= \frac{1}{2} \inner{S_i}{W - U} \\
		\frac{\partial f}{\partial W} &= \sum_{i=1}^n \frac{t_i + 1}{2} S_i \\
		\frac{\partial f}{\partial U} &= \sum_{i=1}^n \frac{1 - t_i}{2} S_i
	\end{align*}
	It is easy to verify that 
	\begin{align*}
		& f(t, W, U) - f(\widetilde{t}, \widetilde{W}, \widetilde{U}) - \inner{\eta_1(\cdot, \cdot)}{\nabla f(\widetilde{t}, \widetilde{W}, \widetilde{U})} =  \sum_{i=1}^n \frac{1}{2} \inner{S_i}{W + U} + \frac{1}{2} t_i \inner{S_i}{W - U} - \\
		&\sum_{i=1}^n \frac{1}{2} \inner{S_i}{\widetilde{W} + \widetilde{U}} - \frac{1}{2} \widetilde{t}_i \inner{S_i}{\widetilde{W} - \widetilde{U}} - \sum_{i=1}^n \tau_i \frac{1}{2} \inner{S_i}{\widetilde{W} - \widetilde{U}} - \inner{W - \widetilde{W}}{\sum_{i=1}^n \frac{\widetilde{t}_i + 1}{2} S_i } - \\
		&\inner{U - \widetilde{U}}{\sum_{i=1}^n \frac{1 - \widetilde{t}_i}{2} S_i } \\
		&= 0
	\end{align*}
\end{proof}

\end{document}